\documentclass[11pt,  a4paper]{amsart}
\usepackage{eucal}
\usepackage{graphicx}
\usepackage[english]{babel}
\usepackage{amsmath,amsthm,amssymb,amsfonts,amscd,amsbsy,amsxtra}
\usepackage{epsfig,color}
\usepackage{subfigure}
\usepackage{psfrag}
\usepackage[arrow,matrix, curve]{xy}
\usepackage{amsmath}
\usepackage{enumerate}
\usepackage[T1]{fontenc} 
\usepackage[utf8]{inputenc}
\usepackage{mathrsfs}
\usepackage[shortlabels]{enumitem}
\usepackage{mathtools}
\usepackage[bookmarks=true,pdfstartview=FitH, pdfborder={0 0 0}, colorlinks=true,citecolor=red, linkcolor=blue]{hyperref}
\usepackage{stmaryrd}
\usepackage{bbm}
\usepackage{aliascnt}
\usepackage{a4wide}
\usepackage{nicefrac}

\renewcommand {\S}	{\mathbb{S}}
\newcommand {\R}	{\mathbb{R}}
\newcommand {\N}	{\mathbb{N}}
\newcommand {\Z}	{\mathbb{Z}}

\newcommand {\C}	{\mathbb{C}}

\DeclareMathOperator{\re}{Re}

\DeclareMathOperator{\Id}{Id}

\DeclareMathOperator{\dvol}{dvol}

\newcommand{\rC}{\mathrm{C}}

\newcommand{\ddn}{\frac{{\partial}^a}{\partial n}}
\newcommand{\tddn}{\tfrac{{\partial}^a}{\partial n}}

\renewcommand{\epsilon}{\varepsilon}

\setlist[enumerate]{font = \normalfont}

\theoremstyle{plain}
\newtheorem{thm}{Theorem}[section]
\newaliascnt{cor}{thm}
\newaliascnt{prop}{thm}
\newaliascnt{lem}{thm}
\newtheorem{cor}[cor]{Corollary}
\newtheorem{prop}[prop]{Proposition}
\newtheorem{lem}[lem]{Lemma}
\aliascntresetthe{cor}
\aliascntresetthe{prop}
\aliascntresetthe{lem} 
\newcounter{stp}
\newcounter{stpi}
\newcounter{stpci}
\newcounter{stpiii}

 \theoremstyle{theorem}

\theoremstyle{definition}
\newaliascnt{defn}{thm}
\newaliascnt{asu}{thm}
\newaliascnt{con}{thm}
\aliascntresetthe{defn}
\aliascntresetthe{asu}
\aliascntresetthe{con}
\theoremstyle{remark}
\newaliascnt{rem}{thm}
\newaliascnt{exa}{thm}
\newaliascnt{masu}{thm}
\newaliascnt{nota}{thm}
\newaliascnt{sett}{thm}
\newtheorem{rem}[rem]{Remark}

\aliascntresetthe{rem}
\aliascntresetthe{exa}
\aliascntresetthe{masu}
\aliascntresetthe{nota}
\aliascntresetthe{sett}
\numberwithin{equation}{section}

\title [Elliptic operators with Dirichlet boundary conditions]
{Strictly elliptic operators with Dirichlet boundary conditions on spaces of continuous functions on manifolds}
\author{Tim Binz}

\date{\today}%

\begin{document}

\begin{abstract}
We study strictly elliptic differential operators 
with Dirichlet boundary conditions on the space $\mathrm{C}(\overline{M})$
of continuous functions on a compact Riemannian manifold $\overline{M}$ with boundary and prove sectoriality 
with optimal angle $\nicefrac{\pi}{2}$. 
\end{abstract}

\maketitle

\section{Introduction}

Our starting point is a smooth compact Riemannian manifold $\overline{M}$ of dimension $n$ with
smooth boundary $\partial M$
and Riemannian metric $g$ and the 
initial value-boundary problem
\begin{align}
\begin{cases}
\frac{d}{dt} u(t) &= \sqrt{|a|} \text{div}_g \left( \frac{1}{\sqrt{|a|}} a \nabla_M^g u(t)\right) +
\langle b, \nabla_M^g u(t) \rangle + c u(t) \phantom{0} \text{ for } t > 0, \\
u(t)|_{\partial M}  &= 0 \phantom{ \sqrt{|a|} \text{div}_g \left( \frac{1}{\sqrt{|a|}} a \nabla_M^g u(t)\right) +
\langle b, \nabla_M^g u(t) \rangle + c u(t) } \text{ for } t > 0, \\
u(0) &= u_0 .
\end{cases} \label{I-B-P} \tag{IBP}
\end{align}
Here $a$ is a smooth $(1,1)$-tensorfield, $b \in C(\overline{M},\R^n)$ and $c \in C(\overline{M},\R)$. 
We are interested in existence, uniqueness and qualitative behaviour of the solution of this initial value-boundary problem. 
To study these properties systematically, the theory of operator semigroups (cf.  \cite{Ama:95}, \cite{EN:00}, \cite{Eva:98}, \cite{Lun:95}) can be used.
We choose the Banach space $\mathrm{C}(\overline{M})$ and define the
\emph{differential operator with Dirichlet boundary condition}
\begin{align*}
A_0 f &:= \sqrt{|a|} \text{div}_g \left( \frac{1}{\sqrt{|a|}} a \nabla_M^g f \right) +
\langle b, \nabla_M^g u(t) \rangle + c f \\
\text{ with domain } \\
D(A_0) &:= \left\{ f \in \bigcap_{p \geq 1} W^{2,p}(M) \cap \rC_0(M) \colon A_0 f \in \mathrm{C}(\overline{M}) \right\} .
\end{align*}
Then the initial value-boundary problem \eqref{I-B-P} is equivalent to the abstract Cauchy problem
\begin{align}
\begin{cases}
\frac{d}{dt} u(t) &= A_0 u(t) \text{ for } t > 0, \\
u(0) &= u_0  
\end{cases} \label{ACP} \tag{ACP}
\end{align}
in
$\mathrm{C}(\overline{M})$. 
In this paper we show that the solution $u$ of the above problems can be extended analytically in the time variable $t$ to the open complex right half-plane. In operator theoretic terms this corresponds to the fact that $A_0$ 
is sectorial of angle $\nicefrac{\pi}{2}$. 
Here is our main theorem.

\begin{thm}\label{mthm}
The operator $A_0$ is sectorial of angle $\nicefrac{\pi}{2}$ and has compact resolvent on $\mathrm{C}(\overline{M})$. 
\end{thm}

For domains $\Omega \subset \R^n$ the generation of analytic semigroups by elliptic operators with Dirichlet boundary conditions on different spaces is well known. 
It was first shown by Browder in \cite{Bro:61} for $L^2(\Omega)$, by Agmon in \cite{Agm:62} for $L^p(\Omega)$ (see also \cite[Chap.~3.1.1]{Lun:95}) and by Stewart in \cite{Ste:74} for $\mathrm{C}(\overline{\Omega})$
(see also \cite[Chap.~3.1.5]{Lun:95}). By Stewart's method one
even gets the optimal angle of analyticity $\nicefrac{\pi}{2}$.  
Using standard localization arguments the result of \cite{Ste:74} can be extend to elliptic operators on manifolds with boundary (\autoref{mthm}), see
\hyperref[sec:localisation]{Section \ref*{sec:localisation}}. Unfortunately, using this method we do not know how the resolvents look like.

\bigskip 

The angle $\nicefrac{\pi}{2}$ of analyticity of $A_0$ plays an important role in the generation of analytic semigroups by elliptic
differential operators with Wentzell boundary conditions
on spaces of continuous functions. Many authors are interested in this topic, and we refer, e.g., to \cite{CM:98}, \cite{FGGR:02}, \cite{Eng:03}, 
\cite{EF:05}, \cite{FGGR:10}. 
In this context one starts from the ``maximal'' operator $A_m:D(A_m)\subseteq\rC(\overline{M})\to\rC(\overline{M})$ in divergence form, given by
\begin{align*}
A_m f &:= \sqrt{|a|} \text{div}_g \left( \frac{1}{\sqrt{|a|}} a \nabla_M^g f\right) +
\langle b, \nabla_M^g f \rangle + c f \\
\text{ with domain } \\
D(A_m) &:= \left\{ f \in \bigcap_{p \geq 1} W^{2,p}(M) \colon A_m f \in \mathrm{C}(\overline{M}) \right\} .
\end{align*}
Moreover, using the outer co-normal derivative $\ddn:D(\ddn)\subset\mathrm{C}(\overline{M})\to\mathrm{C}(\partial M)$, a constant $\beta<0$ and $\gamma\in\mathrm{C}(\partial M)$, one defines the differential operator $A$ with \emph{generalized Wentzell boundary conditions} by requiring
\begin{equation}\label{eq:bc-W-Lap} 
f\in D(A)
\quad:\iff\quad
f \in D(A_m) \text{ and }
A_m f\big|_{\partial M}=\beta\cdot \tddn f+\gamma\cdot f\big|_{\partial M} .
\end{equation}
The main theorem in \cite{BE:18} shows that this operator $A$ can be
splitted into the operator $A_0$ with Dirichlet boundary conditions on $\rC(\overline{M})$ and the \emph{Dirichlet-to-Neumann operator} $N := \beta\cdot \tddn L_0$ on $\rC(\partial M)$, where $L_0 \varphi = f$ denotes the unique solution of
\begin{equation*}
\begin{cases}
A_m f=0,\\
f|_\partial M=\varphi.
\end{cases}
\end{equation*}
Using \autoref{mthm} and \cite[Thm.~3.1~\& ~Cor.~3.2]{BE:18} 
one obtains the following result.

\begin{cor}
The operator $A$ with Wentzell boundary conditions generates a compact and analytic semigroup of angle $\theta > 0$ on $\rC(\overline{M})$ if and only if the Dirichlet-to-Neumann operator $N$ does so on $\rC(\partial M)$. 
\end{cor}

In an upcoming paper \cite{Bin:18} we prove the latter statement with the optimal angle $\nicefrac{\pi}{2}$ and conclude that elliptic differential operators with Wentzell boundary conditions generate compact and
analytic semigroups of angle $\nicefrac{\pi}{2}$ on $\mathrm{C}(\overline{M})$. 

\pagebreak

This paper is organized as follows. 

\smallskip 

In Section 2 we study the special case where $A_0$ is the Laplace-Beltrami operator with Dirichlet boundary conditions. We approximate its resolvents by modifying the Green's functions of the Laplace operator on $\R^n$, study the scaling of the error of the 
Laplace-Beltrami operator and prove estimates for the associated Green's functions.
Finally, one obtains the sectoriality of angle $\nicefrac{\pi}{2}$ for
the Laplace-Beltrami operator with Dirichlet boundary conditions
on $\mathrm{C}(\overline{M})$.

\smallskip

In Section 3 the main result from Section 2 is extended to arbitrary
strictly elliptic operators. Introducing a new
Riemannian metric, induced by the coefficients of the second order part of the elliptic operator, the operator takes a simpler
form: Up to a relatively bounded perturbation of bound $0$, it is 
a Laplace-Beltrami operator for the new metric. 
Regularity and perturbation theory yield the main theorem in its full generality. 

\smallskip 

In Section 4 we give a more elementary proof of our main theorem using standard localisation arguments. 

\bigskip

In this paper the following notation is used. 
For a closed operator $T \colon D(T) \subset X \rightarrow X$ on a Banach space $X$ we denote by $[D(T)]$ the Banach space $D(T)$ 
equipped with the graph norm $\| \bullet \|_T := \| \bullet \|_X + \| T (\bullet) \|_X$ and indicate by $\hookrightarrow$ a continuous
and by $\stackrel{c}{\hookrightarrow}$ a compact embedding. 
Moreover we use Einstein's notation of sums, i.e., 
\begin{align*}
x_k y^k := \sum_{k = 1}^n x_k y^k
\end{align*}
for $x := (x_1, \dots, x_n), y := (y_1, \dots, y_n)$. Furthermore we denote by $\R_+ := \{ r \in \R : r > 0\}$ the positive real numbers and by $\R_- := \R \setminus \R_+$ the non-positive real numbers. Besides one defines the sector by $\Sigma_{\theta} := \{ z \in \C \setminus \{0\} : |\text{arg}(z)| < \theta \}$. Using the distance function $d$ on $\overline{M}$ we denote by $B_R(x) := \{ y \in \overline{M} : d(x,y) < R \}$.

\section*{Acknowledgement}

The author wishes to thank Professor Simon Brendle and Professor Klaus Engel for important suggestions and fruitful discussions. 
Moreover the author thanks the referee for his many helpful comments. 

\section{Laplace-Beltrami operators with Dirichlet boundary conditions}

In this section we consider the special case where $A_0$ is the Laplace-Beltrami operator with Dirichlet boundary conditions, i.e.,
\begin{align}
\Delta^g_0 f &:= \Delta^g f = \text{div}_g(\nabla^g f) = g^{ij} \partial^2_{ij} f - g^{ij} \ ^g\Gamma^{k}_{ij} \partial_k f , \notag \\
D(\Delta^g_0) &:= \left\{ f \in \bigcap_{p \geq 1} W^{2,p}(M) \cap \mathrm{C}_0(\overline{M}) \colon \Delta^g f \in \mathrm{C}(\overline{M}) \right\}
\label{Def:Delta}
\end{align}
on the space $\mathrm{C}(\overline{M})$ of continuous functions on $\overline{M}$.
Here 
\begin{align*} 
^g \Gamma^k_{ij} := \frac{1}{2} g^{kl} \left(
\partial_i g_{jl} + \partial_j g_{il} - \partial_l g_{ij}
\right)
\end{align*}
denote the Christoffel symbols of the Riemannian metric $g$.

\begin{prop}\label{injective}
For all $\lambda \in \C \setminus \R_-$ the operator $\lambda - \Delta^g_0$ is injective.
\end{prop}
\begin{proof}
Considering the equation
\begin{align}
\begin{cases}
\lambda f = \Delta^g f , \\
f|_{\partial M} = 0 , \label{eq:1}
\end{cases}
\end{align}
for $f \in C(\overline{M})$,
one obtains, that $\lambda - \Delta^g_0$ is injective if the only solution of \eqref{eq:1} is zero.



Since $\overline{M}$ is compact, the domain $D(\Delta^g_0)$ is contained in $L^2(M)$ and $\Delta^g f \in L^2(M)$.
Hence, Green's formula implies
\begin{align*}
\lambda\| f \|_{L^2(M)}^2 = \lambda \int_{M} f \overline{f} \, \dvol_M^g 
= \int_M \Delta^g f \overline{f} \, \dvol_M^g 
= -\int_M g(\nabla^g f, \nabla^g \overline{f}) \, \dvol_M^g \in \R_- .
\end{align*}
Since $\lambda \in \C \setminus \R_-$, the term $\lambda \| f \|_{L^2(M)}^2$ can be in $\R_-$ only if $f = 0$. 

%
%
\end{proof}

In the next step we construct Green's functions such that the associated integral operators approximate the resolvent of $A_0$.

\bigskip

To this end, it is necessary to smooth the distance function $d$ on $\overline{M}$. We consider a sufficiently small $\epsilon > 0$ and define
\begin{align*}
\rho(x,y) := d(x,y) \chi\left(\frac{d(x,y)}{\epsilon}\right) + 2\epsilon \left(1-\chi\left(\frac{d(x,y)}{\epsilon}\right)\right) ,
\end{align*}
where $\chi$ is a smooth cut-off function with $\chi(s) = 1$ if $s < 1$
and $\chi(s) = 0$ if $s > 2$.
Then $\rho \equiv d$ for $d(x,y) < \epsilon$ and $\rho \in C^\infty((\overline{M}\times \overline{M})\setminus \{(x,x) : x \in \overline{M} \} ,\R)$.

\smallskip

Next we extend the smoothed distance function $\rho$ on $\overline{M}$ beyond the 
boundary $\partial M$. To this end, the 
set $S_{2\epsilon} := \{ x \in \overline{M} \colon d(x,\partial M) < 2\epsilon \}$ 
is identified via the normal exponential map with $\partial M \times 
[0, 2\epsilon)$. Considering $\overline{M} \cup (\partial M \times (- 2\epsilon,0])$
and identifying $\partial M$ with $\partial M \times \{0\}$ via 
$x \sim (x,0)$, one obtains a smooth manifold $\widetilde{M}$.
By Whitney's extension theorem (see \cite{See:64}) the metric $g$ can be extended to a smooth metric $\bar{g}$ on $\widetilde{M}$ and
hence the smoothed distance function $\rho$ can be extended to a smooth function $\bar{\rho}$ on $\widetilde{M} \times \widetilde{M} \setminus \{(x,x) : x \in \widetilde{M} \}$.

\smallskip

For $x \in S_{2\epsilon}$ we consider the reflected point $x^\ast \in \widetilde{M} \setminus M$ with
\begin{align*}
\bar{\rho}(x, \partial M) = \bar{\rho}(x^\ast, \partial M)
\end{align*}
such that the nearest neighbour of $x$ on $\partial M$ and the nearest neighbour of $x^\ast$ on $\partial M$ coincide. 

\smallskip

Here and in the following we denote by $n := \dim(\overline{M})$ the dimension of the manifold. 
The kernels are defined by
\begin{align*}
K_\lambda(x,y) 
:= 
\begin{cases}
\frac{\sqrt{\lambda}^{\nicefrac{n}{2}-1}}{\sqrt{2\pi}^n}
\left( \frac{K_{\nicefrac{n}{2}-1}(\sqrt{\lambda} \bar{\rho}(x,y))}{\bar{\rho}(x,y)^{\nicefrac{n}{2}-1}}
- \frac{K_{\nicefrac{n}{2}-1}(\sqrt{\lambda} \bar{\rho}(x^\ast,y))}{\bar{\rho}(x^\ast,y)^{\nicefrac{n}{2}-1}} \right) 
, &\text{if} \quad d(x,\partial M) < \epsilon, \\
\frac{\sqrt{\lambda}^{\nicefrac{n}{2}-1}}{\sqrt{2\pi}^n}
\left( \frac{K_{\nicefrac{n}{2}-1}(\sqrt{\lambda} \bar{\rho}(x,y))}{\bar{\rho}(x,y)^{\nicefrac{n}{2}-1}}
- \chi\left(\frac{\bar{\rho}(x,\partial M)}{\epsilon}\right) \frac{K_{\nicefrac{n}{2}-1}(\sqrt{\lambda} \bar{\rho}(x^\ast,y))}{\bar{\rho}(x^\ast,y)^{\nicefrac{n}{2}-1}} \right) 
, &\text{if} \quad d(x,\partial M) \in [\epsilon,2 \epsilon], \\
\frac{\sqrt{\lambda}^{\nicefrac{n}{2}-1}}{\sqrt{2\pi}^n}
\frac{K_{\nicefrac{n}{2}-1}(\sqrt{\lambda} \bar{\rho}(x,y))}{\bar{\rho}(x,y)^{\nicefrac{n}{2}-1}}, &\text{if} \quad d(x,\partial M) > 2\epsilon,
\end{cases}
\end{align*}
for $x \in \tilde{M}, y \in M$ and $\lambda \in \C \setminus \R_-$, 
where $K_{\nicefrac{n}{2}-1}$ is the modified Bessel function of the second kind (cf. \autoref{App:def}) of order $\nicefrac{n}{2}-1$.
Moreover, the associated integral operators are given by
\begin{align*}
(G_{\lambda}  f)(x) := \int_M K_\lambda(x,y) f(y) \, dy .
\end{align*} 

We now prove that the integral operators $G_\lambda$ satisfy similar estimates as the resolvents of a sectorial operator. 

\begin{prop}\label{Greensfunction estimate}
Let $\eta > 0$.
For $\lambda \in \Sigma_{\pi-\eta}$ with $|\lambda| \geq 1$, 
the integral operators
$G_\lambda$ fulfil 
\begin{align*}
\| G_{\lambda}  f \|_{L^\infty(M)} \leq \frac{C(\eta)}{|\lambda|} \| f \|_{L^\infty(M)}
\end{align*}
for all $f \in \mathrm{C}(\overline{M})$ and $C(\eta) > 0$. 
\end{prop}
\begin{proof}
By \autoref{Lem. A.1} and \autoref{Lem. A.2} we obtain
\begin{align}
\| G_{\lambda}  f \|_{L^\infty(M\setminus S_{2\epsilon})} 
\leq C \sqrt{|\lambda|}^{\nicefrac{n}{2}-1}
\sup_{x \in M \setminus S_{2\epsilon}}
\int_{M}  
\frac{K_{\nicefrac{n}{2}-1}(C(\eta) \sqrt{|\lambda|} {\rho}(x,y) )}{{\rho}(x,y)^{\nicefrac{n}{2}-1}}  \, dy \cdot
\|f \|_{L^\infty(M)} \\
 \leq \frac{C'(\eta)}{|\lambda|} \| f \|_{L^\infty(M)}  \notag
\end{align}
for $f \in \mathrm{C}(\overline{M})$. 
Moreover, \autoref{Lem. A.1}, \autoref{Lem. A.2} and \autoref{Cor. A.3} imply
\begin{align}
\label{eq:2}  
\| G_{\lambda}  f \|_{L^\infty(S_{\epsilon})} 
&\leq C \sqrt{|\lambda|}^{\nicefrac{n}{2}-1} 
\left( 
\sup_{x \in S_{{\epsilon}}}
\int_{M}
\frac{K_{\nicefrac{n}{2}-1}(C(\eta)\sqrt{|\lambda|} \overline{\rho}(x,y) )}{\overline{\rho}(x,y)^{\nicefrac{n}{2}-1}} \, dy \right. 
 \\
&+ \left.
\sup_{x \in S_{\epsilon}} 
\int_{M}
\frac{K_{\nicefrac{n}{2}-1}(C(\eta)\sqrt{|\lambda|} \overline{\rho}(x^\ast,y) )}{\overline{\rho}(x^\ast,y)^{\nicefrac{n}{2}-1}} \, dy
\right) 
\|f \|_{L^\infty(M)} \notag \leq \frac{C'(\eta)}{|\lambda|} \| f \|_{L^\infty(M)} 
\end{align}
for $f \in \mathrm{C}(\overline{M})$. 
Furthermore \autoref{Lem. A.1}, \autoref{Lem. A.2} and \autoref{Cor. A.3} yield
\begin{align}
\label{eq:3}  
&\phantom{\leq}\| G_{\lambda}  f \|_{L^\infty(S_{2\epsilon}\setminus S_{\epsilon})} \\
&\leq C \sqrt{|\lambda|}^{\nicefrac{n}{2}-1} 
\left( 
\sup_{x \in S_{2\epsilon}\setminus S_{\epsilon}}
\int_{M}
\frac{K_{\nicefrac{n}{2}-1}(C(\eta)\sqrt{|\lambda|} \overline{\rho}(x,y) )}{{\rho}(x,y)^{\nicefrac{n}{2}-1}} \, dy \right. 
\notag
\\
&+ \left.
\underbrace{
\sup_{x \in S_{2\epsilon}\setminus S_{\epsilon}} 
\chi
\left(
\frac{\overline{\rho}(x,\partial M)}{\epsilon}
\right)
}_{= 1}
\sup_{x \in S_{2\epsilon}\setminus S_{\epsilon}} 
\int_{M}
\frac{K_{\nicefrac{n}{2}-1}(C(\eta)\sqrt{|\lambda|} \overline{\rho}(x^\ast,y) )}{\overline{\rho}(x^\ast,y)^{\nicefrac{n}{2}-1}} \, dy 
\right) 
\|f \|_{L^\infty(M)} \notag \\
&\leq \frac{C'(\eta)}{|\lambda|} \| f \|_{L^\infty(M)} \notag
\end{align}
for $f \in \mathrm{C}(\overline{M})$.
Summing up it follows that
\begin{align*}
\| G_{\lambda} f \|_{L^\infty(M)} 
= \| G_{\lambda} f \|_{L^\infty(M\setminus S_{2\epsilon})} 
+ \| G_{\lambda} f \|_{L^\infty(S_{\epsilon})}  
+ \| G_{\lambda} f \|_{L^\infty(S_{2\epsilon} \setminus S_{\epsilon})} 
\leq \frac{C(\eta)}{|\lambda|} \| f \|_{L^\infty(M)} 
\end{align*}
for $f \in \mathrm{C}(\overline{M})$ as claimed. 
\end{proof}

To show that the kernel $K_\lambda$ is approximately a Green's function for $\lambda - \Delta^g_0$, we need the following lemmata.

\begin{lem}\label{Lem. 1}
Let $\eta > 0$.
For $\lambda \in \Sigma_{\pi-\eta}$ with $|\lambda| \geq 1$, we have
\begin{align*}
(\lambda &- \Delta^{g}_x) 
\left( \frac{\sqrt{\lambda}^{\nicefrac{n}{2}-1}}{\sqrt{2 \pi}^n} 
\frac{K_{\nicefrac{n}{2}-1}(\sqrt{\lambda} \rho(x,y))}{\rho(x,y)^{\nicefrac{n}{2}-1}} \right) \\
&= \delta_x(y) 
+ \mathcal{O}\left( 
\sqrt{|\lambda|}^{\nicefrac{n}{2}-1}
\left( \frac{\sqrt{|\lambda|}K_{\nicefrac{n}{2}}(C(\eta)\sqrt{|\lambda|}\rho(x,y))}{\rho(x,y)^{\nicefrac{n}{2}-2}} +  \frac{K_{\nicefrac{n}{2}-1}(C(\eta)\sqrt{|\lambda|}\rho(x,y))}{\rho(x,y)^{\nicefrac{n}{2}-1}}
\right) + e^{-C(\eta) \sqrt{|\lambda|}\varepsilon} 
\right) 
\end{align*}
for $x,y \in \overline{M}$. 
\end{lem}
\begin{proof}
Considering 
\begin{align}
K(r) := \frac{\sqrt{\lambda}^{\nicefrac{n}{2}-1}}{\sqrt{2\pi}^n} \frac{K_{\nicefrac{n}{2}-1}(\sqrt{\lambda} r)}{r^{\nicefrac{n}{2}-1}}
\label{kernel rotational}
\end{align}
one obtains 
\begin{align}
K'(r) &= \frac{\sqrt{\lambda}^{\nicefrac{n}{2}-1}}{\sqrt{2\pi}^n}
\left( \frac{\sqrt{\lambda}K'_{\nicefrac{n}{2}-1}(\sqrt{\lambda}r)}{r^{\nicefrac{n}{2}-1}}
- \frac{(\nicefrac{n}{2}-1)K_{\nicefrac{n}{2}-1}(\sqrt{\lambda}r)}{r^{\nicefrac{n}{2}}}
\right) \label{eq:K'} \\
&= -\frac{\sqrt{\lambda}^{\nicefrac{n}{2}-1}}{\sqrt{2\pi}^n}
\left( \frac{\sqrt{\lambda}K_{\nicefrac{n}{2}}(\sqrt{\lambda}r)}{2r^{\nicefrac{n}{2}-1}}
+  \frac{\sqrt{\lambda}K_{\nicefrac{n}{2}-2}(\sqrt{\lambda}r)}{2r^{\nicefrac{n}{2}-1}}
+ \frac{(\nicefrac{n}{2}-1)K_{\nicefrac{n}{2}-1}(\sqrt{\lambda}r)}{r^{\nicefrac{n}{2}}}
\right) \notag
\end{align}
and hence
\begin{align}
K''(r) = \frac{\sqrt{\lambda}^{\nicefrac{n}{2}-1}}{\sqrt{2\pi}^n}
\left( \frac{\lambda K''_{\nicefrac{n}{2}-1}(\sqrt{\lambda}r)}{r^{\nicefrac{n}{2}-1}} - (n-2) \frac{ \sqrt{\lambda} K'_{\nicefrac{n}{2}-1}(\sqrt{\lambda} r)}{r^{\nicefrac{n}{2}}} +
(\nicefrac{n^2}{4}-\nicefrac{n}{2}) \frac{K_{\nicefrac{n}{2}-1}(\sqrt{\lambda} r)}{r^{\nicefrac{n}{2}+1}}  \right) . \label{eq:K''}
\end{align}
These imply
\begin{align*}
&\phantom{=}K''(r) + \frac{n-1}{r} K'(r) - \lambda K(r) \\
&= \frac{\sqrt{\lambda}^{\nicefrac{n}{2}-1}}{\sqrt{2\pi}^n}
\left( \frac{\lambda K''_{\nicefrac{n}{2}-1}(\sqrt{\lambda}r)}{r^{\nicefrac{n}{2}-1}} + \frac{ \sqrt{\lambda} K'_{\nicefrac{n}{2}-1}(\sqrt{\lambda} r)}{r^{\nicefrac{n}{2}}} +
(\nicefrac{n^2}{4}-\nicefrac{n}{2}) \frac{K_{\nicefrac{n}{2}-1}(\sqrt{\lambda} r)}{r^{\nicefrac{n}{2}+1}} \right. \\
&-  (n-1)(\nicefrac{n}{2}-1) \frac{K_{\nicefrac{n}{2}-1}(\sqrt{\lambda}r)}{r^{\nicefrac{n}{2}+1}} - \left. \lambda \frac{K_{\nicefrac{n}{2}-1}(\sqrt{\lambda} r)}{r^{\nicefrac{n}{2}-1}}
\right) \\
&= \frac{\sqrt{\lambda}^{\nicefrac{n}{2}-1}}{\sqrt{2\pi}^n r^{\nicefrac{n}{2}+1}} \left( 
\lambda r^2 K_{\nicefrac{n}{2}-1}''(\sqrt{\lambda} r) + \sqrt{\lambda} r K_{\nicefrac{n}{2}-1}'(\sqrt{\lambda} r)
- ( (\nicefrac{n}{2}-1)^2 + \lambda r^2 ) K_{\nicefrac{n}{2}-1}(\sqrt{\lambda} r) \right) .
\end{align*}
Remark that the kernel is rotation symmetric and hence the Laplacian is given by
$\Delta_x K(|x|) = \partial_r^2 K(|x|) + \frac{n-1}{r} \partial_r K(|x|)$. Using \eqref{def:Bessel} we conclude
\begin{align}
\Delta_x K(|x|) - \lambda K(|x|) = -C \cdot \delta_0(x) . \label{K Greens Fct}
\end{align}
Next we determining the constant $C$. For sufficient small $R > 0$ one has by Gauss Divergence Theorem
\begin{align*}
C = \int_{B_R(0)} C \cdot \delta_0(x) \, \dvol_{B_R(0)} &= - \int_{B_R(0)} \Delta_x K(|x|) - \lambda K(|x|) \, \dvol_{B_R(0)} \\
&= - \int_{\S^{n-1}_R} \frac{\partial}{\partial n} K(|x|) \, \dvol_{S^{n-1}_R}
+ \lambda \int_{B_R(0)} K(|x|) \, \dvol_{B_R(0)}\\
&= - \int_{\S^{n-1}_R} K'(R) \, \dvol_{S^{n-1}_R}
+ \lambda \int_{B_R(0)} K(|x|) \, \dvol_{B_R(0)} \\
&= - \text{vol}(S^{n-1}) K'(R) R^{n-1}
+ \lambda \int_{B_R(0)} K(|x|) \, \dvol_{B_R(0)} .
\end{align*}
Using \eqref{eq:K'} we obtain
\begin{align*}
C &=
\frac{\sqrt{\lambda}^{\nicefrac{n}{2}-1}}{\sqrt{2\pi}^n} \text{vol}(S^{n-1}) 
\left( \frac{\sqrt{\lambda}K_{\nicefrac{n}{2}}(\sqrt{\lambda}R)}{2R^{\nicefrac{n}{2}-1}}
+  \frac{\sqrt{\lambda}K_{\nicefrac{n}{2}-2}(\sqrt{\lambda}R)}{2R^{\nicefrac{n}{2}-1}}
+ \frac{(\nicefrac{n}{2}-1)K_{\nicefrac{n}{2}-1}(\sqrt{\lambda}R)}{R^{\nicefrac{n}{2}}}
\right)
 R^{n-1} \\
&+ \lambda \int_{B_R(0)} K(|x|) \, \dvol_{B_R(0)} .
\end{align*}
Since $K_\alpha (r) = \mathcal{O}(r^{-\alpha})$ for small $r \in \R_+$, the second and the fourth term vanishes by taking the limit $R \to 0$. 
Since 
\begin{equation*}
\lim_{r \to 0} r^\alpha K_\alpha(r) = 2^{\alpha-1} \Gamma(\alpha) 
\end{equation*}
and 
\begin{align*}
\text{vol}(\S^{n-1}) = n \cdot \frac{\sqrt{\pi}^n}{\Gamma(\nicefrac{n}{2}+1)}
\end{align*}
the limit of the first term is given by
\begin{align*}
\frac{\sqrt{\lambda}^{\nicefrac{n}{2}}}{2^{\nicefrac{n}{2}+1} \sqrt{\pi}^n } \cdot \text{vol}(\S^{n-1}) \cdot \lim_{R \to 0} K_{\nicefrac{n}{2}}(\sqrt{\lambda}R)\cdot R^{\nicefrac{n}{2}} 
&= 
\frac{\sqrt{\lambda}^{\nicefrac{n}{2}}}{2^{\nicefrac{n}{2}+1}} \cdot \frac{n}{\Gamma(\nicefrac{n}{2}+1)} \lim_{R \to 0} K_{\nicefrac{n}{2}}(\sqrt{\lambda}R)\cdot R^{\nicefrac{n}{2}} \\
&= \frac{1}{2^{\nicefrac{n}{2}+1}} \cdot \frac{n}{\Gamma(\nicefrac{n}{2}+1)} \lim_{R' \to 0} K_{\nicefrac{n}{2}}(R')\cdot (R')^{\nicefrac{n}{2}}
\\
&= \frac{1}{2^{\nicefrac{n}{2}+1}} \cdot \frac{n}{\Gamma(\nicefrac{n}{2}+1)} \cdot
2^{\nicefrac{n}{2}-1} \Gamma(\nicefrac{n}{2}) \\
&= \frac{1}{2^{\nicefrac{n}{2}+1}} \cdot \frac{2}{\Gamma(\nicefrac{n}{2})} \cdot
2^{\nicefrac{n}{2}-1} \Gamma(\nicefrac{n}{2})
= \frac{1}{2},
\end{align*}
where we used in the last line, that the Gamma function satisfies $\Gamma(x+1) = x \Gamma(x)$. 

Similar the limit of the third term is 
\begin{align*}
&\phantom{=}(\nicefrac{n}{2}-1) \cdot \frac{\sqrt{\lambda}^{\nicefrac{n}{2}-1}}{\sqrt{2\pi}^n } \cdot \text{vol}(\S^{n-1}) \cdot \lim_{R \to 0} K_{\nicefrac{n}{2}-1}(\sqrt{\lambda}R) \cdot R^{\nicefrac{n}{2}-1} \\
&= (\nicefrac{n}{2}-1) \cdot \frac{\sqrt{\lambda}^{\nicefrac{n}{2}-1}}{\sqrt{2}^n } \cdot \frac{n}{\Gamma(\nicefrac{n}{2}+1)} \cdot \lim_{R \to 0} K_{\nicefrac{n}{2}-1}(\sqrt{\lambda}R) \cdot R^{\nicefrac{n}{2}-1} \\
&= (\nicefrac{n}{2}-1) \cdot \frac{1}{\sqrt{2}^n} \cdot \frac{n}{\Gamma(\nicefrac{n}{2}+1)} \cdot \lim_{R' \to 0} K_{\nicefrac{n}{2}-1}(R') \cdot (R')^{\nicefrac{n}{2}-1} \\
&=  (\nicefrac{n}{2}-1) \cdot \frac{1}{\sqrt{2}^n} \cdot \frac{n}{\Gamma(\nicefrac{n}{2}+1)} \cdot 2^{\nicefrac{n}{2}-2} \Gamma(\nicefrac{n}{2}-1) \\
&= (\nicefrac{n}{2}-1) \cdot \frac{1}{\sqrt{2}^n} \cdot \frac{n}{\nicefrac{n}{2} \cdot (\nicefrac{n}{2}-1) \cdot \Gamma(\nicefrac{n}{2}-1)} \cdot 2^{\nicefrac{n}{2}-2} \Gamma(\nicefrac{n}{2}-1) = \frac{1}{2} .
\end{align*}
Hence $C = \nicefrac{1}{2} + \nicefrac{1}{2} = 1$. 
Moreover we have
\begin{align*}
(K \circ \rho) (x,y) = \frac{\sqrt{\lambda}^{\nicefrac{n}{2}-1}}{\sqrt{2\pi}^n} \frac{K_{\nicefrac{n}{2}-1}(\sqrt{\lambda} \rho(x,y))}{\rho(x,y)^{\nicefrac{n}{2}-1}} 
\end{align*}
for $x, y \in \overline{M}$.
Using geodetic normal coordinates the metric is given by
\begin{align*}
g_{ij}(x) = \delta_{ij} + \mathcal{O}(\rho(x,y)^2).
\end{align*}
From $\Delta^\delta_x |x - y|^2 = 2n$, it follows
\begin{align*}
\Delta^g_x (\rho(x,y)^2) = 2n + \mathcal{O}(\rho(x,y)^2) .
\end{align*}
Using 
\begin{align}
\Delta^g_x (\rho(x,y)^2)
= 2 | \nabla^g_x \rho(x,y) |_g^2 + 2 \rho(x,y) \Delta^g_x \rho(x,y) \label{Laplace rho}
\end{align}
First we consider $y \in B_\varepsilon(x)$. Since $|\nabla^g_x \rho(x,y) |_g = 1$
one obtains 
\begin{align*}
\Delta^g_x (\rho(x,y)) 
= \frac{n-1}{\rho(x,y)} + \mathcal{O}(\rho(x,y)) .
\end{align*}
Therefore, we obtain
\begin{align*}
\Delta^g_x (K \circ \rho)(x,y)
&= K''(\rho(x,y)) |\nabla^g_x \rho(x,y) |_g^2 +
K'(\rho(x,y)) \Delta^g_x \rho(x,y) \\
&= K''(\rho(x,y)) +
K'(\rho(x,y)) \Delta^g_x \rho(x,y) \\
&= K''(\rho(x,y)) + \frac{n-1}{\rho(x,y)} K'(\rho(x,y)) + \mathcal{O}\left(\rho(x,y) |K'(\rho(x,y))| \right) .
\end{align*}
Using \eqref{K Greens Fct} and \autoref{Lem. A.1}, it follows that
\begin{align*}
(\lambda-\Delta^g_x) (K \circ \rho)(x,y)
&= \delta_x (y) + \mathcal{O}\left(\rho(x,y) |K'(\rho(x,y))| \right) \\
&= \delta_x (y) + \mathcal{O}\left( 
\sqrt{|\lambda|}^{\nicefrac{n}{2}-1}
\left( \frac{\sqrt{|\lambda|}K_{\nicefrac{n}{2}}(C(\eta)\sqrt{|\lambda|}\rho(x,y))}{\rho(x,y)^{\nicefrac{n}{2}-2}} \right. \right. \\
&\phantom{= \delta_x(y) }\ + \left. \left.  \frac{(\nicefrac{n}{2}-1) \cdot K_{\nicefrac{n}{2}-1}(C(\eta)\sqrt{|\lambda|}\rho(x,y))}{\rho(x,y)^{\nicefrac{n}{2}-1}}
\right) \right) . 
\end{align*}
Now let $y \in B_{2\varepsilon}(x)\setminus B_\varepsilon(x)$. Since $\rho$ is smooth on $\overline{M}\setminus B_{\varepsilon}(x)$, we have
$|\nabla^g_x \rho(x,y)|_g^2 \leq C$ and therefore by \eqref{Laplace rho}
\begin{align*}
|\Delta^g_x (K \circ \rho)(x,y)| \leq C |K''(\rho(x,y))| + C(n) \frac{|K'(\rho(x,y))|}{\rho(x,y)} + \mathcal{O}(\rho(x,y)|K'(\rho(x,y))|) .
\end{align*}
Moreover one obtains by \autoref{Lem. A.1}
\begin{align*}
|K_{\nicefrac{n}{2}-1}(\sqrt{\lambda} \rho(x,y))| &\leq K_{\nicefrac{n}{2}-1}(C(\eta) \sqrt{|\lambda|} \varepsilon) = \mathcal{O}(e^{-C(\eta) \sqrt{|\lambda|} \varepsilon})\\
|K'_{\nicefrac{n}{2}-1}(\sqrt{\lambda} \rho(x,y))| &\leq  |K_{\nicefrac{n}{2}}(\sqrt{\lambda} \rho(x,y))| + |K_{\nicefrac{n}{2}}(\sqrt{\lambda} \rho(x,y))| \\
&\leq |K_{\nicefrac{n}{2}}(C(\eta)\sqrt{|\lambda|} \varepsilon)| + |K_{\nicefrac{n}{2}}(C(\eta)\sqrt{|\lambda|} \varepsilon)|
= \mathcal{O}(e^{-C(\eta)\sqrt{|\lambda|} \varepsilon})\\
|K'_{\nicefrac{n}{2}-1}(\sqrt{\lambda} \rho(x,y))| &\leq 
|K_{\nicefrac{n}{2}+1}(\sqrt{\lambda} \rho(x,y))| +
|K_{\nicefrac{n}{2}-1}(\sqrt{\lambda} \rho(x,y))| + |K_{\nicefrac{n}{2}-3}(\sqrt{\lambda} \rho(x,y))| \\
&= |K_{\nicefrac{n}{2}+1}(C(\eta)\sqrt{|\lambda|} \varepsilon)| +
|K_{\nicefrac{n}{2}-1}(C(\eta)\sqrt{|\lambda|} \varepsilon)| + |K_{\nicefrac{n}{2}-3}(C(\eta)\sqrt{|\lambda|} \varepsilon)| \\
&= \mathcal{O}(e^{-C(\eta) \sqrt{|\lambda|} \varepsilon})
\end{align*}
for $|\lambda|$and $\lambda \in \Sigma_{\nicefrac{\pi}{2}-\eta}$. Since $\rho(x,y) \geq \varepsilon$, it follows
\begin{align*}
(\lambda - \Delta^g_x) (K \circ \rho)(x,y) = \mathcal{O}(e^{-C(\eta) \sqrt{|\lambda|} \varepsilon})
\end{align*}
for $|\lambda| \geq 1$.

Finally, we consider $y \in \overline{M} \setminus B_{2\varepsilon}(x)$. Since $\rho$ is constant on $\overline{M} \setminus B_{2\varepsilon}(x)$, it follows that $\Delta_x^g (K \circ \rho) = 0$ and therefore as before
\begin{align*}
(\lambda - \Delta^g_x) (K \circ \rho)(x,y) = \mathcal{O}(e^{-C(\eta) \sqrt{|\lambda|} \varepsilon})
\end{align*}
for $|\lambda| \geq 1$.
\end{proof} 

\begin{lem}\label{Lem. 2}
Let $\eta > 0$.
For $\lambda \in \Sigma_{\pi-\eta}$ with $|\lambda| \geq 1$, we have
\begin{align*}
&\phantom{=}(\lambda - \Delta^{g}_x) 
\left( \frac{\sqrt{\lambda}^{\nicefrac{n}{2}-1}}{\sqrt{2 \pi}^n} 
\frac{K_{\nicefrac{n}{2}-1}(\sqrt{\lambda} \overline{\rho}(x^\ast,y))}{\overline{\rho}(x^\ast,y)^{\nicefrac{n}{2}-1}} \right) 
\\
&= \delta_{x^\ast}(y) + \mathcal{O} 
\left(
\sqrt{|\lambda|}^{\nicefrac{n}{2}-1} \frac{
K_{\nicefrac{n}{2}-1}(C(\eta)\sqrt{|\lambda|}\overline{\rho}(x^\ast,y))}{{\overline{\rho}(x^\ast,y)^{\nicefrac{n}{2}}}} \right. \notag \\
&+ \sqrt{|\lambda|}^{\nicefrac{n}{2}} \frac{
K_{\nicefrac{n}{2}}(C(\eta)\sqrt{|\lambda|}\overline{\rho}(x^\ast,y))}{{\overline{\rho}(x^\ast,y)^{\nicefrac{n}{2}-1}}}  \notag \\
&+ 
d(x,\partial M) \left(
\sqrt{|\lambda|}^{\nicefrac{n}{2}-1} \frac{
K_{\nicefrac{n}{2}-1}(C(\eta)\sqrt{|\lambda|}\overline{\rho}(x^\ast,y))}{{\overline{\rho}(x^\ast,y)^{\nicefrac{n}{2}+1}}} \right. \notag \\
&+ \sqrt{|\lambda|}^{\nicefrac{n}{2}} \frac{
K_{\nicefrac{n}{2}}(C(\eta)\sqrt{|\lambda|}\overline{\rho}(x^\ast,y))}{{\overline{\rho}(x^\ast,y)^{\nicefrac{n}{2}}}} \notag \\
&+ \left. \left. \sqrt{|\lambda|}^{\nicefrac{n}{2}+1} \frac{
K_{\nicefrac{n}{2}+1}(C(\eta)\sqrt{|\lambda|}\overline{\rho}(x^\ast,y))}{{\overline{\rho}(x^\ast,y)^{\nicefrac{n}{2}-1}}}
\right) \notag
%
+ e^{-C \sqrt{|\lambda|}\varepsilon}
\right)
\end{align*}
for $x \in S_{2\epsilon}$ and $y \in \overline{M}$. 
\end{lem}
\begin{proof}
Considering the reflection $\sigma \colon S_{2\epsilon} \rightarrow \tilde{M}
\colon x \mapsto x^\ast$ and taking a point on the boundary $p \in \partial M$ every normal vector is an eigenvector for the
eigenvalue $-1$ for the differential $D\sigma_p: T_p M \to T_p M$ and all tangential vectors on $\partial M$ eigenvectors with eigenvalue $1$. In particular $D\sigma_p$ is a linear local isometry, i.e. $\sigma^\ast g = g$ for $p \in \partial M$. Since $\sigma^\ast g - g$ is smooth, we conclude that
\begin{align*}
\sigma^\ast g = g + \mathcal{O}(d(x,\partial M)).
\end{align*}
Hence, one obtains
$\nabla_x^\delta \sigma^\ast g = \nabla_x^\delta g + \mathcal{O}(1)$ and
\begin{align*}
\Delta^g \tilde{h} = \Delta^g (h \circ \sigma) = (\Delta^{\sigma^\ast g} h) \circ \sigma
\end{align*}
for $\tilde{h} (x) := h(x^\ast)$. Therefore, 
\begin{align}
(\lambda - \Delta^g) \tilde{h} = ((\lambda - \Delta^{\sigma^\ast g})h) \circ \sigma
+ (\Delta^{\sigma^\ast g} h - \Delta^g h) \circ \sigma . \label{eq:Spiegelung}
\end{align}
Using 
\begin{equation*}
\Delta^g f = g^{ij} (\partial_{ij}^2 f- \Gamma^k_{ij} \partial_k f)
\end{equation*}
we obtain
\begin{equation*}
| \Delta^g f - \Delta^{\sigma^\ast g} f|(x) \leq C \cdot | g - \sigma^\ast g|_g (x)
\cdot \sum_{i,j = 1}^n |\partial_{ij}^2 f |(x)+ C \cdot \left| \nabla g - \nabla (\sigma^\ast g)\right|_g (x) \cdot | \nabla f |_g(x) .
\end{equation*}
Since $| g - \sigma^\ast g|_g (x) = \mathcal{O}(d(x,\partial M))$ and $\left| \nabla g - \nabla (\sigma^\ast g)\right|_g(x) = \mathcal{O}(1)$ we consider the derivatives of the kernel. 
Define $K$ as in \eqref{kernel rotational} we obtain
\begin{align*}
\partial_i (K \circ \overline{\rho})(x^\ast,y) &= K'(\overline{\rho}(x^\ast,y)) \cdot \partial_i \overline{\rho}(x^\ast,y) \\
\partial_{ij}^2 (K \circ \overline{\rho})(x^\ast,y) &= K''(\overline{\rho}(x^\ast,y)) \cdot \partial_i \overline{\rho}(x^\ast,y) \cdot \partial_j \overline{\rho}(x^\ast,y) + K'(\overline{\rho}(x^\ast,y)) \cdot \partial_{ij}^2 \overline{\rho}(x^\ast,y) .
\end{align*}
Since $\partial_i \overline{\rho}(x^\ast,y) = \mathcal{O}(1)$ and \eqref{eq:K'} we obtain
\begin{equation*}
\nabla_x (K \circ \overline{\rho})(x^\ast,y) = \mathcal{O}\left( \sqrt{|\lambda|}^{\nicefrac{n}{2}-1} \frac{
K_{\nicefrac{n}{2}-1}(C(\eta)\sqrt{|\lambda|}\overline{\rho}(x^\ast,y))}{{\overline{\rho}(x^\ast,y)^{\nicefrac{n}{2}}}}
+ \sqrt{|\lambda|}^{\nicefrac{n}{2}} \frac{
K_{\nicefrac{n}{2}}(C(\eta)\sqrt{|\lambda|}\overline{\rho}(x^\ast,y))}{{\overline{\rho}(x^\ast,y)^{\nicefrac{n}{2}-1}}}
\right) ,
\end{equation*}
where we used that by \autoref{Lem. A.1} and the monotonicity of Bessel functions
\begin{equation*}
|K_{\nicefrac{n}{2}-2} (\sqrt{\lambda} \overline{\rho}(x^\ast,y))| \leq K_{\nicefrac{n}{2}-2} (C(\eta)\sqrt{|\lambda|} \overline{\rho}(x^\ast,y))
\leq K_{\nicefrac{n}{2}} (C(\eta)\sqrt{|\lambda|} \overline{\rho}(x^\ast,y))
\end{equation*}
holds.
Similar we obtain from \eqref{eq:K''} 
\begin{align*}
K''(\overline{\rho}(x^\ast,y)) &= \mathcal{O}\left(
\sqrt{|\lambda|}^{\nicefrac{n}{2}-1} \frac{
K_{\nicefrac{n}{2}-1}(C(\eta)\sqrt{|\lambda|}\overline{\rho}(x^\ast,y))}{{\overline{\rho}(x^\ast,y)^{\nicefrac{n}{2}+1}}}
+ \sqrt{|\lambda|}^{\nicefrac{n}{2}} \frac{
K_{\nicefrac{n}{2}}(C(\eta)\sqrt{|\lambda|}\overline{\rho}(x^\ast,y))}{{\overline{\rho}(x^\ast,y)^{\nicefrac{n}{2}}}} \right. \\
&\qquad \qquad + \left. \sqrt{|\lambda|}^{\nicefrac{n}{2}+1} \frac{
K_{\nicefrac{n}{2}+1}(C(\eta)\sqrt{|\lambda|}\overline{\rho}(x^\ast,y))}{{\overline{\rho}(x^\ast,y)^{\nicefrac{n}{2}-1}}}
\right) 
\end{align*} 
Using
\begin{equation*}
\partial_{ij}^2 \overline{\rho}(x^\ast,y)^2 = 2 \partial_i \overline{\rho}(x^\ast,y) \partial_j \overline{\rho}(x^\ast,y) + 
2\overline{\rho}(x^\ast,y) \partial_{ij}^2 \overline{\rho}(x^\ast,y)
\end{equation*}
and $\partial_i \overline{\rho}(x^\ast,y) = \mathcal{O}(1)$ and $\partial_{ij}^2 \overline{\rho}(x^\ast,y)^2 = \mathcal{O}(1)$ one has
\begin{equation*}
\partial_{ij}^2 \overline{\rho}(x^\ast,y) = \mathcal{O}\left( \frac{1}{\overline{\rho}(x^\ast,y)} \right) .
\end{equation*}
Hence
\begin{align*}
\partial_{ij}^2 (K \circ \overline{\rho})(x^\ast,y) 
&= \mathcal{O}\left(
\sqrt{|\lambda|}^{\nicefrac{n}{2}-1} \frac{
K_{\nicefrac{n}{2}-1}(C(\eta)\sqrt{|\lambda|}\overline{\rho}(x^\ast,y))}{{\overline{\rho}(x^\ast,y)^{\nicefrac{n}{2}+1}}}
+ \sqrt{|\lambda|}^{\nicefrac{n}{2}} \frac{
K_{\nicefrac{n}{2}}(C(\eta)\sqrt{|\lambda|}\overline{\rho}(x^\ast,y))}{{\overline{\rho}(x^\ast,y)^{\nicefrac{n}{2}}}} \right. \\
&\qquad \qquad + \left. \sqrt{|\lambda|}^{\nicefrac{n}{2}+1} \frac{
K_{\nicefrac{n}{2}+1}(C(\eta)\sqrt{|\lambda|}\overline{\rho}(x^\ast,y))}{{\overline{\rho}(x^\ast,y)^{\nicefrac{n}{2}-1}}}
\right) 
\end{align*}
Finally, we conclude
\begin{align}
&\phantom{=}\left((\Delta^{\sigma^\ast g}_x - \Delta^g_x) \left( 
\frac{\sqrt{\lambda}^{\nicefrac{n}{2}-1}}{\sqrt{2 \pi}^n} 
\frac{K_{\nicefrac{n}{2}-1}(\sqrt{\lambda} \overline{\rho}(\cdot,y))}{\overline{\rho}(\cdot,y)^{\nicefrac{n}{2}-1}}
\right)\right) (x^\ast) \notag \\
&= \mathcal{O} \left(
\sqrt{|\lambda|}^{\nicefrac{n}{2}-1} \frac{
K_{\nicefrac{n}{2}-1}(C(\eta)\sqrt{|\lambda|}\overline{\rho}(x^\ast,y))}{{\overline{\rho}(x^\ast,y)^{\nicefrac{n}{2}}}} \right. \notag \\
&+ \sqrt{|\lambda|}^{\nicefrac{n}{2}} \frac{
K_{\nicefrac{n}{2}}(C(\eta)\sqrt{|\lambda|}\overline{\rho}(x^\ast,y))}{{\overline{\rho}(x^\ast,y)^{\nicefrac{n}{2}-1}}}  \label{eq:Spiegelung-Error} \\
&+ 
d(x,\partial M) \left(
\sqrt{|\lambda|}^{\nicefrac{n}{2}-1} \frac{
K_{\nicefrac{n}{2}-1}(C(\eta)\sqrt{|\lambda|}\overline{\rho}(x^\ast,y))}{{\overline{\rho}(x^\ast,y)^{\nicefrac{n}{2}+1}}} \right. \notag \\
&+ \sqrt{|\lambda|}^{\nicefrac{n}{2}} \frac{
K_{\nicefrac{n}{2}}(C(\eta)\sqrt{|\lambda|}\overline{\rho}(x^\ast,y))}{{\overline{\rho}(x^\ast,y)^{\nicefrac{n}{2}}}} \notag \\
&+ \left. \left. \sqrt{|\lambda|}^{\nicefrac{n}{2}+1} \frac{
K_{\nicefrac{n}{2}+1}(C(\eta)\sqrt{|\lambda|}\overline{\rho}(x^\ast,y))}{{\overline{\rho}(x^\ast,y)^{\nicefrac{n}{2}-1}}}
\right) \right)  . \notag
\end{align}
Now the claim follows by \autoref{Lem. 1} (for $\sigma^\ast g$ instead of $g$), using \eqref{eq:Spiegelung} and \eqref{eq:Spiegelung-Error}.
\end{proof}

\begin{lem}\label{Lem. 3}
Let $\eta > 0$.
We obtain
\begin{align*}
&\phantom{=}(\lambda - \Delta^g_x) \left( 
\frac{\sqrt{\lambda}^{\nicefrac{n}{2}-1}}{\sqrt{2\pi}^n}
\chi\left( 
\frac{\rho(x,\partial M)}{\epsilon} \right)
\frac{K_{\nicefrac{n}{2}-1}(\sqrt{\lambda} \overline{\rho}(x^\ast,y))}{\overline{\rho}(x^\ast,y)^{\nicefrac{n}{2}-1}}
\right) 
= \mathcal{O}\left(
e^{-C(\eta) \sqrt{|\lambda|}\epsilon}\right) 
\end{align*}
for $y \in \overline{M}$, $x \in S_{2\epsilon} \setminus S_{\epsilon}$ and
for $\lambda \in \Sigma_{\pi-\eta}$ with $|\lambda| \geq 1$, 
\end{lem}
\begin{proof}
By the product rule an easy calculation yields
\begin{align*}
&\phantom{=}(\lambda - \Delta^g_x)\left(
\frac{\sqrt{\lambda}^{\nicefrac{n}{2}-1}}{\sqrt{2\pi}^n}
\chi \left( \frac{{\rho}(x,\partial M)}{\epsilon} \right)  
\frac{K_{\nicefrac{n}{2}-1}(\sqrt{\lambda} \overline{\rho}(x^\ast, y))}{\overline{\rho}(x^\ast, y)^{\nicefrac{n}{2}-1}} \right) \\
&= \chi \left( \frac{{\rho}(x,\partial M)}{\epsilon} \right)
(\lambda - \Delta^g_x)\left(
\frac{\sqrt{\lambda}^{\nicefrac{n}{2}-1}}{\sqrt{2\pi}^n}
\frac{K_{\nicefrac{n}{2}-1}(\sqrt{\lambda} \overline{\rho}(x^\ast, y))}{\overline{\rho}(x^\ast, y)^{\nicefrac{n}{2}-1}} \right) \\
&- \Delta^g_x
\left( 
\chi \left( \frac{{\rho}(x,\partial M)}{\epsilon} \right)
\right) 
\frac{\sqrt{\lambda}^{\nicefrac{n}{2}-1}}{\sqrt{2\pi}^n}
\frac{K_{\nicefrac{n}{2}-1}(\sqrt{\lambda} \overline{\rho}(x^\ast, y))}{\overline{\rho}(x^\ast, y)^{\nicefrac{n}{2}-1}} \\
&- 
2 \frac{\sqrt{\lambda}^{\nicefrac{n}{2}-1}}{\sqrt{2\pi}^n} 
\left\langle \nabla^g_x \chi \left( \frac{{\rho}(x,\partial M)}{\epsilon} \right),
\nabla^g_x \left( 
\frac{K_{\nicefrac{n}{2}-1}(\sqrt{\lambda} \overline{\rho}(x^\ast, y))}{\overline{\rho}(x^\ast, y)^{\nicefrac{n}{2}-1}} \right)
\right\rangle . 
\end{align*}
Using \autoref{Lem. 1} and \autoref{Lem. 2} one obtains for the first term
\begin{align*}
&\phantom{=}\chi \left( \frac{{\rho}(x,\partial M)}{\epsilon} \right)(\lambda - \Delta^{g}_x) 
\left( \frac{\sqrt{\lambda}^{\nicefrac{n}{2}-1}}{\sqrt{2 \pi}^n} 
\frac{K_{\nicefrac{n}{2}-1}(\sqrt{\lambda} \overline{\rho}(x^\ast,y))}{\overline{\rho}(x^\ast,y)^{\nicefrac{n}{2}-1}} \right) 
\\
&+ \mathcal{O} 
\left(
\sqrt{|\lambda|}^{\nicefrac{n}{2}-1} \frac{
K_{\nicefrac{n}{2}-1}(C(\eta)\sqrt{|\lambda|}\overline{\rho}(x^\ast,y))}{{\overline{\rho}(x^\ast,y)^{\nicefrac{n}{2}}}} \right. \notag \\
&+ \sqrt{|\lambda|}^{\nicefrac{n}{2}} \frac{
K_{\nicefrac{n}{2}}(C(\eta)\sqrt{|\lambda|}\overline{\rho}(x^\ast,y))}{{\overline{\rho}(x^\ast,y)^{\nicefrac{n}{2}-1}}}  \notag \\
&+ 
d(x,\partial M) \left(
\sqrt{|\lambda|}^{\nicefrac{n}{2}-1} \frac{K_{\nicefrac{n}{2}-1}(C(\eta)\sqrt{|\lambda|}\overline{\rho}(x^\ast,y))}{{\overline{\rho}(x^\ast,y)^{\nicefrac{n}{2}+1}}} \right. \notag \\
&+ \sqrt{|\lambda|}^{\nicefrac{n}{2}} \frac{
K_{\nicefrac{n}{2}}(C(\eta)\sqrt{|\lambda|}\overline{\rho}(x^\ast,y))}{{\overline{\rho}(x^\ast,y)^{\nicefrac{n}{2}}}} \notag \\
&+ \left. \left. \sqrt{|\lambda|}^{\nicefrac{n}{2}+1} \frac{
K_{\nicefrac{n}{2}+1}(C(\eta)\sqrt{|\lambda|}\overline{\rho}(x^\ast,y))}{{\overline{\rho}(x^\ast,y)^{\nicefrac{n}{2}-1}}}
\right) \notag
%
+ e^{-C \sqrt{|\lambda|}\varepsilon}
\right)
\end{align*}
Since $d(x,\partial M) \in [\epsilon,2\epsilon]$ is bounded away from $0$, \autoref{Lem. A.1}
yields
\begin{align*}
\frac{K_{\nicefrac{n}{2}-1}(C(\eta) \sqrt{|\lambda|} \overline{\rho}(x^\ast,y))}{\overline{\rho}(x^\ast,y)^{\nicefrac{n}{2}+1}} 
&\leq  
\frac{K_{\nicefrac{n}{2}-1}(C(\eta) \sqrt{|\lambda|} \epsilon)}{\epsilon^{\nicefrac{n}{2}+1}} .
\end{align*}
Since $d(x,\partial M) < 2 \varepsilon$ and
\begin{align*}
K_\alpha (\sqrt{|\lambda|} \epsilon) = \mathcal{O}( e^{-\sqrt{|\lambda|\epsilon}} ),
\end{align*}
one concludes that
\begin{align*}
&\phantom{=}\chi \left( \frac{{\rho}(x,\partial M)}{\epsilon} \right)(\lambda - \Delta^{g}_x) 
\left( \frac{\sqrt{\lambda}^{\nicefrac{n}{2}-1}}{\sqrt{2 \pi}^n} 
\frac{K_{\nicefrac{n}{2}-1}(\sqrt{\lambda} \overline{\rho}(x^\ast,y))}{\overline{\rho}(x^\ast,y)^{\nicefrac{n}{2}-1}} \right)
= \mathcal{O}\left(
e^{-C(\eta) \sqrt{|\lambda|}\epsilon} \right) 
\end{align*}
for $\lambda \in \Sigma_{\pi-\eta}$ with $|\lambda| \geq 1$. 
Since $|\nabla^g_x \rho|_g$ is bounded on $S_{2\varepsilon} \setminus S_{\varepsilon}$ and
$|\Delta^g_x \rho|_g \leq \frac{C}{\rho}$ on $S_{2\varepsilon} \setminus S_{\varepsilon}$, it follows that
\begin{align*}
\nabla^g_x  \left( \chi\left(\frac{\rho(x,\partial M)}{\epsilon}\right)
\right)
= \chi'\left(\frac{\rho(x,\partial M)}{\epsilon}\right) \frac{\nabla^g_x \rho(x,\partial M)}{\epsilon}
= \mathcal{O}(1)
\end{align*}
and
\begin{align*}
\Delta^g_x \left( \chi\left(\frac{\rho(x,\partial M)}{\epsilon}\right)
\right)
= \chi''\left(\frac{\rho(x,\partial M)}{\epsilon}\right) \frac{|\nabla^g_x \rho(x,\partial M)|^2}{\epsilon^2}
+ \chi'\left(\frac{\rho(x,\partial M)}{\epsilon}\right) \frac{\Delta^\delta_x \rho(x,\partial M)}{\epsilon}
= \mathcal{O}(1) .
\end{align*}
Hence, the second term satisfies
\begin{align*}
\Delta^g_x
\left( 
\chi \left( \frac{{\rho}(x,\partial M)}{\epsilon} \right)
\right) 
\frac{\sqrt{\lambda}^{\nicefrac{n}{2}-1}}{\sqrt{2\pi}^n}
\frac{K_{\nicefrac{n}{2}-1}(\sqrt{\lambda} \overline{\rho}(x^\ast, y))}{\overline{\rho}(x^\ast, y)^{\nicefrac{n}{2}-1}}
= \mathcal{O}\left(
e^{-C(\eta) \sqrt{|\lambda|}\epsilon} \right) 
\end{align*}
for $\lambda \in \Sigma_{\pi - \eta}$ with $|\lambda|\geq 1$. 
Since
\begin{align*}
\nabla^g_x \left( 
\frac{K_{\nicefrac{n}{2}-1}(\sqrt{\lambda} \overline{\rho}(x^\ast, y))}{\overline{\rho}(x^\ast, y)^{\nicefrac{n}{2}-1}} \right)
&=
\frac{\sqrt{\lambda}K'_{\nicefrac{n}{2}-1}(\sqrt{\lambda} \overline{\rho}(x^\ast, y))\nabla^g_{x}\overline{\rho}(x^\ast,y)}{\overline{\rho}(x^\ast, y)^{\nicefrac{n}{2}-1}} \\
&-
(\nicefrac{n}{2}-1)
\frac{K_{\nicefrac{n}{2}-1}(\sqrt{\lambda} \overline{\rho}(x^\ast, y))\nabla^g_{x}\overline{\rho}(x^\ast,y)}{\overline{\rho}(x^\ast, y)^{\nicefrac{n}{2}}} \\ 
&=
\frac{\sqrt{\lambda}K_{\nicefrac{n}{2}}(\sqrt{\lambda} \overline{\rho}(x^\ast, y))\nabla^g_{x}\overline{\rho}(x^\ast,y)}{2\overline{\rho}(x^\ast, y)^{\nicefrac{n}{2}-1}} \\
&+\frac{\sqrt{\lambda}K_{\nicefrac{n}{2}-2}(\sqrt{\lambda} \overline{\rho}(x^\ast, y))\nabla^g_{x}\overline{\rho}(x^\ast,y)}{2\overline{\rho}(x^\ast, y)^{\nicefrac{n}{2}-1}} \\
&-
(\nicefrac{n}{2}-1)
\frac{K_{\nicefrac{n}{2}-1}(\sqrt{\lambda} \overline{\rho}(x^\ast, y))\nabla^g_{x}\overline{\rho}(x^\ast,y)}{\overline{\rho}(x^\ast, y)^{\nicefrac{n}{2}}} \\
&= \mathcal{O}\left(
e^{-C(\eta) \sqrt{|\lambda|}\epsilon} \right),
\end{align*}
we conclude
\begin{align*}
\frac{\sqrt{\lambda}^{\nicefrac{n}{2}-1}}{\sqrt{2\pi}^n} 
\left\langle \nabla^g_x \chi \left( \frac{{\rho}(x,\partial M)}{\epsilon} \right),
\nabla^g_x \left( 
\frac{K_{\nicefrac{n}{2}-1}(\sqrt{\lambda} \overline{\rho}(x^\ast, y))}{\overline{\rho}(x^\ast, y)^{\nicefrac{n}{2}-1}} \right)
\right\rangle
= \mathcal{O}\left(
e^{-C(\eta) \sqrt{|\lambda|}\epsilon} \right) 
\end{align*}
for $\lambda \in \Sigma_{\pi - \eta}$ for $|\lambda|\geq 1$.
Summing up the claim follows. 
\end{proof}

Now we are prepared to show that $K_{\lambda}$ is approximately a Green's function for $\lambda - \Delta^g_x$. 

\begin{thm}\label{Error estimate}
The integral operators $G_\lambda$ satisfy
\begin{align*}
\| (\lambda- \Delta^g_x)G_\lambda f - f \|_{L^\infty(M)} \leq \frac{C(\eta)}{\sqrt{|\lambda|}} \| f \|_{L^\infty(M)}
\end{align*}
for $\lambda \in \Sigma_{\pi-\eta}$
with $|\lambda|\geq 1$, $\eta > 0$, and $f \in \mathrm{C}(\overline{M})$. 
\end{thm}
\begin{proof}
For $x \in \overline{M} \setminus S_{2\epsilon}$  \autoref{Lem. 1} yields
\begin{align*}
\| (\lambda - \Delta^g_x) G_{\lambda} f - f \|_{L^\infty(M \setminus S_{2\epsilon})}
&\leq \sup_{ x \in \overline{M} \setminus S_{2\epsilon} }
\left| \int_M \delta_x(y) f(y) \, dy - f(x) \right| \\
&+ \mathcal{O} 
\left( \left( 
\sup_{ x \in \overline{M} \setminus S_{2\epsilon} }
\sqrt{|\lambda|}^{\nicefrac{n}{2}}
\int_M   \frac{K_{\nicefrac{n}{2}}(C(\eta)\sqrt{|\lambda|}\rho(x,y))}{\rho(x,y)^{\nicefrac{n}{2}-2}} \, dy \right. \right.  \\
&+ \sup_{ x \in \overline{M} \setminus S_{2\epsilon} }
\sqrt{|\lambda|}^{\nicefrac{n}{2}-1} \int_M \frac{K_{\nicefrac{n}{2}-1}(C(\eta)\sqrt{|\lambda|}\rho(x,y))}{\rho(x,y)^{\nicefrac{n}{2}-1}} 
\, dy  \\
&+ \left. \left. \int_M e^{-C(\eta) \sqrt{|\lambda|}\varepsilon} \, dy 
\right) \| f \|_{L^\infty(M)} \right)
\end{align*}
for $\lambda \in \Sigma_{\pi-\eta}$ with $|\lambda|\geq 1$ and $f \in \mathrm{C}(\overline{M})$. 
Therefore, by \autoref{Lem. A.2} it follows that
\begin{align*}
\| (\lambda - \Delta^g_x) G_{\lambda} f - f \|_{L^\infty(M\setminus S_{2\epsilon})} 
\leq \frac{C(\eta)}{\sqrt{|\lambda|}} \| f \|_{L^\infty(M)} 
\end{align*}
for $\lambda \in \Sigma_{\pi-\eta}$ with $|\lambda|\geq 1$ and $f \in \mathrm{C}(\overline{M})$. For $x \in S_{\epsilon}$ we obtain by \autoref{Lem. 1} and \autoref{Lem. 2} 
\begin{align*}
\| (\lambda - \Delta^g_x) G_{\lambda} f - f \|_{L^\infty(S_{\epsilon})}
&\leq \sup_{ x \in S_{\epsilon} }
\left| \int_M \delta_x(y) f(y) \, dy - f(x) \right| \\
&+ \mathcal{O} 
\left( \left( 
\sup_{ x \in S_{\epsilon} }
\sqrt{|\lambda|}^{\nicefrac{n}{2}}
\int_M   \frac{K_{\nicefrac{n}{2}}(C(\eta)\sqrt{|\lambda|}\rho(x,y))}{\rho(x,y)^{\nicefrac{n}{2}-2}} \, dy \right. \right.  \\
&+ \sup_{ x \in S_{\epsilon} }
\sqrt{|\lambda|}^{\nicefrac{n}{2}-1} \int_M \frac{K_{\nicefrac{n}{2}-1}(C(\eta)\sqrt{|\lambda|}\rho(x,y))}{\rho(x,y)^{\nicefrac{n}{2}-1}} 
\, dy  \\
&+ \sup_{ x \in S_{\epsilon} }
\sqrt{|\lambda|}^{\nicefrac{n}{2}}
\int_M   \frac{K_{\nicefrac{n}{2}}(C(\eta)\sqrt{|\lambda|}\overline{\rho}(x^\ast,y))}{\overline{\rho}(x^\ast,y)^{\nicefrac{n}{2}-2}} \, dy \\
&+ \sup_{ x \in S_{\epsilon} }
\sqrt{|\lambda|}^{\nicefrac{n}{2}-1} \int_M \frac{K_{\nicefrac{n}{2}-1}(C(\eta)\sqrt{|\lambda|}\overline{\rho}(x^\ast,y))}{\overline{\rho}(x^\ast,y)^{\nicefrac{n}{2}-1}} 
\, dy  \\
&+ \sup_{ x \in S_{\epsilon} }
\sqrt{|\lambda|}^{\nicefrac{n}{2}-1}
\int_M  d(x,\partial M)
\frac{K_{\nicefrac{n}{2}-1}(C(\eta)\sqrt{|\lambda|} \overline{\rho}(x^\ast,y))}{\overline{\rho}(x^\ast,y)^{\nicefrac{n}{2}+1}} \, dy \\
&+ \sup_{ x \in S_{\epsilon} }
\sqrt{|\lambda|}^{\nicefrac{n}{2}}
\int_M  d(x,\partial M)
\frac{K_{\nicefrac{n}{2}}(C(\eta)\sqrt{|\lambda|} \overline{\rho}(x^\ast,y))}{\overline{\rho}(x^\ast,y)^{\nicefrac{n}{2}}} \, dy \\
&+ \sup_{ x \in S_{\epsilon} }
\sqrt{|\lambda|}^{\nicefrac{n}{2}-1}
\int_M  d(x,\partial M)
\frac{K_{\nicefrac{n}{2}+1}(C(\eta)\sqrt{|\lambda|} \overline{\rho}(x^\ast,y))}{\overline{\rho}(x^\ast,y)^{\nicefrac{n}{2}+1}} \, dy \\
&+ \left. \left. \int_M e^{-C(\eta) \sqrt{|\lambda|}\varepsilon} \, dy 
\right) \| f \|_{L^\infty(M)} \right)
\end{align*}
for $f \in \mathrm{C}(\overline{M})$. 
Since $\overline{\rho}(x^\ast,y)$ only vanish if $x, y \in \partial M$ and $d(x,\partial M) = d(x^\ast, \partial M)
\leq \overline{\rho}(x^\ast,y)$ for $x, y$ near $\partial M$, \autoref{Lem. A.2} and \autoref{Cor. A.3} imply 
\begin{align*}
\|(\lambda - \Delta^g_x) G_{\lambda} f - f \|_{L^\infty(S_{2\epsilon})} 
\leq \frac{C(\eta)}{\sqrt{|\lambda|}} \| f \|_{L^\infty(M)} 
\end{align*}
for $|\lambda|$and $f \in \mathrm{C}(\overline{M})$.
Moreover, we have for $x \in S_{2\epsilon} \setminus S_{\epsilon}$ by \autoref{Lem. 1} and \autoref{Lem. 3}
\begin{align*}
\| (\lambda - \Delta^g_x) G_{\lambda} f - f \|_{L^\infty(S_{2\epsilon} \setminus S_{\epsilon})}
&\leq \sup_{ x \in S_{2\epsilon} \setminus S_{\epsilon} }
\left| \int_M \delta_x(y) f(y) \, dy - f(x) \right| \\
&+ \mathcal{O} 
\left( \left( 
\sup_{ x \in S_{2\epsilon} \setminus S_{\epsilon} }
 \sqrt{|\lambda|}^{\nicefrac{n}{2}} \int_M 
\frac{K_{\nicefrac{n}{2}}(C(\eta)\sqrt{|\lambda|} \rho(x,y))}{\rho(x,y)^{\nicefrac{n}{2}-2}} \, dy  \right. \right. \\
&+
\sup_{ x \in S_{2\epsilon} \setminus S_{\epsilon} }
 \sqrt{|\lambda|}^{\nicefrac{n}{2}-1} \int_M 
\frac{K_{\nicefrac{n}{2}-1}(C(\eta)\sqrt{|\lambda|} \rho(x,y))}{\rho(x,y)^{\nicefrac{n}{2}-1}} \, dy \\
&+ 
\left. \left.  \int_M e^{-C \sqrt{|\lambda|}\epsilon} \, dy 
\right) \| f \|_{L^\infty(M)} \right) 
\end{align*}
for $f \in \mathrm{C}(\overline{M})$. 

Since 
$\overline{M}$ is compact, it follows that
\begin{align*}
\int_M \frac{ e^{-C(\eta) \sqrt{|\lambda|}\epsilon} }{\epsilon^{\nicefrac{n}{2}+1}} \, 
dy  \leq \frac{\tilde{C}(\eta)}{\sqrt{|\lambda|}}
\end{align*}
for $|\lambda| \geq 1$. Hence, as a consequence of \autoref{Lem. A.2} one obtains
\begin{align*}
\| G_{\lambda} f - f \|_{L^\infty(S_{2\epsilon} \setminus S_{\epsilon})} 
\leq \frac{C(\eta)}{\sqrt{|\lambda|}} \| f \|_{L^\infty(M)} 
\end{align*}
for $|\lambda|$and $f \in \mathrm{C}(\overline{M})$.
Summing up we conclude that
\begin{align*}
\|(\lambda - \Delta^g_x) G_{\lambda} f - f \|_{L^\infty(M)}
\leq \frac{C(\eta)}{\sqrt{|\lambda|}} \| f \|_{L^\infty(M)} 
\end{align*}
for $|\lambda|$and $f \in \mathrm{C}(\overline{M})$.
\end{proof}

Finally, we obtain the main theorem by combining
the estimates from \autoref{Greensfunction estimate} and \autoref{Error estimate}.

\begin{thm}\label{main thm}
The operator $\Delta^g_0$ is sectorial of angle $\nicefrac{\pi}{2}$ on $\mathrm{C}(\overline{M})$. 
\end{thm}
\begin{proof}
For $\lambda \in \Sigma_{\pi-\eta}$ with sufficient large absolute value $|\lambda|$ \autoref{Error estimate} implies that
\begin{align*}
\| (\lambda - \Delta^g) G_{\lambda} - \Id \|
\leq \frac{C(\eta)}{\sqrt{|\lambda|}} < 1 , 
\end{align*}
hence $(\lambda - \Delta^g)G_{\lambda}$ is invertible. 
Therefore
\begin{align*}
\Id = (\lambda - \Delta^g) G_{\lambda} ((\lambda - \Delta^g)G_{\lambda})^{-1}
\end{align*}
and $(\lambda - \Delta^g)$ is right-invertible with right-inverse
\begin{align*}
(\lambda - \Delta^g)^{-1} = G_{\lambda} ((\lambda - \Delta^g)G_{\lambda})^{-1}.
\end{align*} 
Hence, by \autoref{injective} the operator $(\lambda - \Delta^g)$ is invertible
and 
\begin{align*}
(\lambda - \Delta^g)^{-1} = G_{\lambda} ((\lambda - \Delta^g)G_{\lambda})^{-1}.
\end{align*} 
In particular, we obtain
\begin{align*}
\Delta^g G_{\lambda} ((\lambda - \Delta^g)G_{\lambda})^{-1} f = \lambda G_{\lambda} ((\lambda - \Delta^g)G_{\lambda})^{-1} f - f \in \mathrm{C}(\overline{M}) 
\end{align*}
for all $f \in \mathrm{C}(\overline{M})$. 
Moreover $G_{\lambda} ((\lambda - \Delta^g)G_{\lambda})^{-1} f$ is a solution of 
\begin{align*}
\begin{cases}
\Delta_x^g u = \lambda u -f, \\
u|_{\partial M} = 0   
\end{cases}
\end{align*}
for $\lambda \in \Sigma_{\pi-\eta}$ with sufficient large absolute value $|\lambda|$. Since $f \in \rC(M) \subset L^p(M)$ for every $p \geq 1$, elliptic regularity (cf. \cite[Thm.~8.12]{GT:01}) implies $G_{\lambda} ((\lambda - \Delta^g)G_{\lambda})^{-1} f \in \bigcap_{p \geq 1}W^{2,p}(M)$. Therefore $G_{\lambda} ((\lambda - \Delta^g)G_{\lambda})^{-1} f \in D(A_0)$ and one concludes $R(\lambda, \Delta^g_0) = G_{\lambda} ((\lambda - \Delta^g)G_{\lambda})^{-1}$ for $\lambda \in \Sigma_{\pi-\eta}$ with sufficient large absolute value $|\lambda|$.
Thus by \autoref{Greensfunction estimate} it follows that
\begin{align*}
\| R(\lambda, \Delta^g_0) \| \leq \| G_{\lambda} \| \cdot \| ((\lambda - \Delta^g)G_\lambda)^{-1} \| \leq \frac{C(\eta)}{|\lambda|} 
\end{align*} 
for $\lambda \in \Sigma_{\pi-\eta}$ with sufficient large absolute value $|\lambda|$.
By \cite[Thm.~3.7.11]{ABHN:01} and \cite[Cor.~3.7.17]{ABHN:01}, $\Delta^g_0$ is sectorial of angle $\nicefrac{\pi}{2}$. 
\end{proof} 

\section{Strict elliptic operators with Dirichlet boundary conditions}

In this section we consider strictly elliptic second-order differential operators
with Dirichlet boundary conditions on the space $\mathrm{C}(\overline{M})$ of the continuous functions for a smooth, compact, Riemannian manifold $(\overline{M},g)$
with smooth boundary $\partial M$. 
To this end, take real-valued functions 
\begin{equation*}
a_j^k = a_k^j \in \rC^{\infty}(\overline{M})
, \quad b_j, c \in \mathrm{C}(\overline{M}), \quad 1 \leq j,k \leq n. 
\end{equation*}
satisfying the strict ellipticity condition 
\begin{equation*}
a^k_j(q) g^{jl}(q) X_k(q) X_l(q) > 0 \quad \text{ for all } q \in \overline{M} 
\end{equation*}
for all co-vectorfields $X_k, X_l$ on $\overline{M}$ with $(X_1(q),\dots, X_n(q)) \not = (0, \dots, 0)$ 
and define on $\mathrm{C}(\overline{M})$ the differential operator in divergence form with Dirichlet boundary conditions as
\begin{align}
A_0 f &:=
\sqrt{|a|}\text{div}_g \left(\frac{1}{\sqrt{|a|}} a \nabla_M^g f \right) + \langle b, \nabla_M^g f \rangle + c f \label{Def:A_0 M} \\
\text{with domain} \notag \\
D(A_0) &:= \left\{ f \in \bigcap_{p \geq 1} W^{2,p}(M) \cap C_0(\overline{M}) \colon A_0 f \in \mathrm{C}(\overline{M}) \right\}, \notag
\end{align}
where $a = a^k_j$, $|a|=\det(a^k_j)$ and $b = (b_1, \dots, b_n)$.

\smallskip

The key idea is to reduce the strictly elliptic operator on $\overline{ M}$, equipped by $g$, to the Laplace-Beltrami operator on $\overline{M}$, corresponding to a new metric $\tilde{g}$. 

\smallskip

For this purpose we consider a $(2,0)$-tensorfield on $\overline{M}$ given by
\begin{align*}
\tilde{g}^{kl} = a^k_i g^{il} .
\end{align*}
Its inverse $\tilde{g}$ is a $(0,2)$-tensorfield on $\overline{M}$, which is a Riemannian metric since $a^k_j g^{jl}$ is strictly elliptic on $\overline{M}$. 
We denote $\overline{M}$ with the old metric by $\overline{M}^g$ and
with the new metric by $\overline{M}^{\tilde{g}}$ and remark that
$\overline{M}^{\tilde{g}}$ is a smooth, compact, orientable Riemannian manifold with smooth boundary $\partial M$. 
Since the differentiable structures of $\overline{M}^g$ and $\overline{ M}^{\tilde{g}}$ coincide, the identity
\begin{equation*}
\Id \colon \overline{M}^g \longrightarrow \overline{M}^{\tilde{g}}
\end{equation*}
is a $C^\infty$-diffeomorphism. Hence,
the spaces
\begin{align*}
C(\overline{M}) &:= C(\overline{M}^{\tilde{g}}) = C(\overline{M}^g) 
\end{align*} 
coincide.
Moreover, \cite[Prop. 2.2]{Heb:00} implies that the spaces
\begin{align}
L^p(M) &:= L^p(M^{\tilde{g}}) = L^p(M^g),\label{Sobolev} \\
W^{k,p}(M) &:= W^{k,p}(M^{\tilde{g}}) = W^{k,p}(M^g), \notag 
\end{align} 
for all $p \geq 1$ and $k \in \N$ coincide.
We now denote by $\Delta^{\tilde{g}}_0$ the operator defined as in 
\eqref{Def:Delta} respecting $\tilde{g}$. Moreover we denote by $\tilde{A}_0$ the operator given in \eqref{Def:A_0 M} for $b_k = c = 0$. 

\begin{lem}\label{Stoerung}
	The operator $A_0$ and $\tilde{A}_0$ differ only by a relatively bounded perturbation of bound $0$.
\end{lem}
\begin{proof}
	Consider
	\begin{align*}
	Pf :=  g^{kl} b_k \partial_l f + c f 
	\end{align*}
	for $f \in D(A_0) \cap D(\tilde{A}_0)$.
Since $D(\tilde{A}_0)$ is contained in $\bigcap_{p > 1} W^{2,p}(M)$, Morreys embedding (cf. \cite[Chap.~V.~and~Rem.~5.5.2]{Ada:75}) and the closed graph theorem imply
	\begin{align}
	[D(\tilde{A}_0)] \stackrel{c}{\hookrightarrow} C^1(\overline{M}) \hookrightarrow \mathrm{C}(\overline{M}),
	\label{Embeddings}
	\end{align}
	in particular $D(\tilde{A}_0)$ and $D(A_0)$ coincide. 
	Since $P \in \mathcal{L}(C^1(\overline{M}), \mathrm{C}(\overline{M}))$ and it follows by \eqref{Embeddings} and Ehrling's Lemma
	(see \cite[Thm.~6.99]{RR:93}) that $P$ is relatively $\tilde{A}_0$-bounded with bound $0$.
\end{proof}

\begin{lem}\label{LB}
The operator $\tilde{A}_0$ equals the Laplace-Beltrami operator $\Delta^{\tilde{g}}_0$ with respect to $\tilde{g}$. 
\end{lem}
\begin{proof}
Using \eqref{Sobolev}, we calculate in local coordinates
\begin{align*}
\tilde{A}_0 f &= \frac{1}{\sqrt{|g|}} \sqrt{|a|} \partial_j \left(\sqrt{|g|} \frac{1}{\sqrt{|a|}} a_l^j g^{kl} \partial_k f\right) \\
&= \frac{1}{\sqrt{|\tilde{g}|}} \partial_j \left(\sqrt{|\tilde{g}|} \tilde{g}^{kl} \partial_k f\right) 
\end{align*}
for $f \in D(\tilde{A}_0)$, since $|g| = |a| \cdot |\tilde{g}|$.
\end{proof}

\begin{thm}\label{main theorem general form}
The operator $A_0$ is sectorial of angle $\nicefrac{\pi}{2}$ on $\mathrm{C}(\overline{M})$. 
\end{thm}
\begin{proof}
By \autoref{main thm} and \autoref{LB} it follows that $\tilde{A}_0$ generates an analytic semigroup of angle $\nicefrac{\pi}{2}$ on $\mathrm{C}(\overline{M})$. Finally \autoref{Stoerung} and
\cite[Thm.~III.~2.10]{EN:00} implies the claim. 
\end{proof}

\begin{rem}
This generalizes \cite[Cor.~3.1.21.(ii)]{Lun:95} to manifolds with boundary.
\end{rem}

By \autoref{main theorem general form} the abstract Cauchy problem \eqref{ACP} is well-posed.
This implies the existence and uniqueness of a continuous solution $u$ of the initial value-boundary problem \eqref{I-B-P},
having an analytic extension in a right half space in the time variable. Moreover, $u(t), A_0 u(t) \in C^\infty(M) \cap \mathrm{C}(\overline{M})$ for all $t > 0$. 

\begin{cor}\label{compactness}
The resolvents $R(\lambda,A_0)$ are compact operators for all $\lambda \in \rho(A_0)$. 
\end{cor}
\begin{proof}
This follows immediately by \eqref{Embeddings} and \cite[Prop.~II.~4.25]{EN:00}.
\end{proof}

We finish this section with the special case of closed manifolds, i.e. $\partial M = \emptyset$. Then the Dirichlet boundary conditions gets an empty condition. Hence the operator $A_0$ becomes
\begin{align*}
A f &:=
\sqrt{|a|} \text{div}_g \left( \frac{1}{\sqrt{|a|}} a \nabla_M^g f\right) + \langle b, \nabla_M^g f \rangle + c f, \notag \\
\text{with domain} \\
D(A) &:= \left\{ f \in \bigcap_{p \geq 1} W^{2,p}(M) \colon A_0 f \in \mathrm{C}(M) \right\} \notag .
\end{align*}
Remark that then $d(x,\partial M) = d(x,\emptyset) = \infty$ and the kernel $K_\lambda$ becomes much easier. 

\begin{cor}
If the manifold $M$ is closed, the operator $A$ generates a compact and analytic semigroup of angle $\nicefrac{\pi}{2}$ on $\rC(M)$.
\end{cor}
\begin{proof}
Since $\rC^2(M) \subset D(A)$ and $\rC^2(M) \subset \rC(M)$ dense, it follows that $A$ is densely defined. Now \autoref{main theorem general form} and \cite[Thm.~III.4.6]{EN:00} imply that $A$ generates an
analytic semigroup of angle $\nicefrac{\pi}{2}$ on $\rC(M)$. Finally, the compactness of the semigroup follows by \autoref{compactness} and \cite[Thm.~II.4.29]{EN:00}. 
\end{proof}

\section{Strictly elliptic operators with Dirichlet boundary conditions II}
\label{sec:localisation}

In this section we give a more elementary and easier proof of \autoref{mthm} following from the analogous results for elliptic operators on domains just by the definitions of the considered objects. 

\smallskip 

We consider a smooth, compact Riemannian manifold $(\overline{M},g)$ with smooth boundary $\partial M$ and a strictly elliptic differential operator $A_0$ in divergence form with Dirichlet boundary conditions given by \eqref{Def:A_0 M}.
Now the resolvent problem
\begin{equation}
	\lambda u - A_0 u = f \text{ on } \overline{M}
	\label{eq:res}
\end{equation}
can be written in local coordinates.
Since $\overline{M}$ is compact there exists a finite atlas $(U_i,\phi_i)$ with $i = 1, \dots, r$, i.e. $\overline{M} = \bigcup\limits_{i = 1}^r U_i$ and $\phi_i \colon U_i \to V_i \subset \R^n_+$ are diffeomorphisms. 
Let $(\chi_i)_{i = 1,\dots,r}$ be the partition of unity subordinated to the covering $(U_i)_{i \in {1,\dots,r}}$, i.e. $\chi_i \in \rC^\infty_c(U_i,\R)$, $\chi_i(q) \in [0,1]$ and $\sum_{i = 1}^r \chi_i = 1$.
Denote by
$u_i := \chi_i \cdot u$, $f_i := \chi_i \cdot f$, $\tilde{u}_i := u_i \circ \phi_i^{-1}$, $\tilde{f}_i = f_i \circ \phi_i^{-1}$.
Since $u|_{\partial M} = 0$ and $\chi_i$ vanishes on $\overline{M} \setminus U_i$ we get 
$u_i\in \bigcap\limits_{p \geq 1} \mathrm{W}^{2,p}(U_i) \cap \rC_0(U_i)$, $f_i \in \rC(U_i)$
and
$\tilde{u}_i \in \bigcap\limits_{p \geq 1} \mathrm{W}^{2,p}(V_i) \cap \rC_0(V_i)$, $\tilde{f}_i \in \rC(V_i)$.
Further let $A_i :=  $
Since by definition $A u = \sum_{i = 1}^r A_i u_i$ we obtain 
\begin{equation*}
	\begin{cases}	
		\lambda u_i - A_i u_i = f_i \text{ on } U_i \\
		\phantom{.}\qquad u_i = 0 \text{ on } \partial U_i 
	\end{cases}
\end{equation*}
for all $i \in \{1,\dots,r\}$.
Using 
\begin{align*}
	\tilde{a}^{ij} &:= (g^{ik}a^{j}_k\partial_i \chi \partial_j \chi)\circ \phi^{-1} \\
	\tilde{b}^k &:= \left( g^{ik}a_k^j\partial_{ij}^2 \chi^k + \left( \sqrt{\frac{|a|}{|g|}} \partial_i \left(  \sqrt{\frac{|g|}{|a|}} g^{kl} a^i_k \right) + b^l \right) \partial_l \chi_k \right)  \circ \phi^{-1} \\
	\tilde{c} &:= c \circ \phi^{-1}
\end{align*}
and $\tilde{A}_i := \tilde{a}^{ij} \partial_{ij}^2 + \tilde{b}^k \partial_k + \tilde{c}$ we obtain
\begin{equation}
	( A_i u_i ) \circ \phi^{-1} = \tilde{A}_i \tilde{u}_i 
	\text{ on } V_i
	\label{eq:res local M}  
\end{equation}
and therefore 
\begin{equation}
	\begin{cases}	
		\lambda \tilde{u}_i - \tilde{A}_i \tilde{u}_i = \tilde{f}_i \text{ on } V_i \\
		\phantom{.}\qquad
		\tilde{u}_i = 0 \text{ on } \partial V_i 
	\end{cases}
	\label{eq:res local}
\end{equation}
for $i \in \{ 1, \dots, r \}$.
Thus our resolvent problem \eqref{eq:res} is by definition just the collection of finite many resolvent problems for strictly elliptic operator $\tilde{A}_i$ on bounded sets $V_i$. 
Since every (second-countable) manifold (with boundary) admits an adequate atlas we can even choose $V_i = B_1(0) \subset \R^n_+$.
Now it follows from \cite[Cor. 3.1.21 (ii)]{Lun:95}
that \eqref{eq:res local} admits a unique solution for all $\lambda \in \Sigma_{\nicefrac{\pi}{2}-\eta}\cap \{ \re(\lambda) > \Lambda_i \}$ for constants $\Lambda_i$ and $\eta > 0$. Let $\Lambda := \max_{i = 1}^r \Lambda_i$.
We conclude that all equations \eqref{eq:res local M} have unique solutions $\tilde{u}_i \in \bigcap\limits_{p \geq 1} \mathrm{W}^{2,p}(U_i) \cap \rC_0(U_i)$ for $\lambda \in \Sigma_{\nicefrac{\pi}{2}-\eta}\cap \{ \re(\lambda) > \Lambda \}$ for $\eta > 0$.
This implies that \eqref{eq:res} has a unique solution $u \in D(A_0)$
for $\lambda \in \Sigma_{\nicefrac{\pi}{2}-\eta}\cap \{ \re(\lambda) > \Lambda \}$ for $\eta > 0$.
Further, since $|\lambda-\Lambda| \cdot \|\tilde{u}_i\|_{\rC(V_i)} \leq \| \tilde{f}_i \|_{\rC(V_i)}$ we conclude 
$|\lambda-\Lambda| \cdot \|{u}_i\|_{\rC(U_i)} \leq \| {f}_i \|_{\rC(U_i)}$ and
$|\lambda| \cdot \|u\|_{\rC(\overline{M})} \leq \| f \|_{\rC(\overline{M})}$
for $\lambda \in \Sigma_{\nicefrac{\pi}{2}-\eta}\cap \{ \re(\lambda) > \Lambda \}$ for $\eta > 0$, i.e.
$A_0$ is sectorial of angle $\nicefrac{\pi}{2}$ on $\rC(\overline{M})$.


\appendix

\section{Bessel functions}

The solutions of the ordinary differential equation 
\begin{align}
z^2 \frac{d^2}{dz^2} f(z) + z \frac{d}{dz} f(z) = (z^2 + \alpha^2) f(z) 
\label{def:Bessel}
\end{align}
for $z \in \C$ are called \emph{modified Bessel functions of order $\alpha \in \R$}. In particular we have the following. 

\begin{prop}\label{App:def}
The 
modified Bessel functions of first kind of order $\alpha \in \R$
are given by
\begin{align*}
I_\alpha(z) = \sum_{k = 0}^\infty 
\frac{ \left( \frac{z}{2} \right)^{2k + \alpha} }{\Gamma( k + \alpha + 1 )k!}
\end{align*} 
for $z \in \C \setminus \R_-$, where $\Gamma$ denotes the Gamma function. 
Moreover we obtain the 
\emph{modified Bessel function of second kind of order $\alpha \in \R\setminus \Z$}
by
\begin{align*}
K_\alpha (z) = \frac{\pi}{2} \cdot \frac{I_{-\alpha}(z) - I_\alpha(z) }{\sin( \pi \alpha )}
\end{align*}
for $z \in \C \setminus \R_-$. If $\alpha \in \Z$, there exists a sequence $(\alpha_n)_{n \in \N} \subset \R \setminus \Z$ such that $\alpha_n \to \alpha$ and $K_\alpha$ is the limit
\begin{align*}
K_\alpha (z) := \lim\limits_{n \to \infty} K_{\alpha_n} (z)
\end{align*}
for $z \in \C \setminus \R_-$.
\end{prop}

First we prove an estimate for the modified Bessel function of second kind. 

\begin{lem}\label{Lem. A.1}
Let $\alpha \in \R$ and $\eta > 0$. 
Then there exists a constant $C(\eta) > 0$ such that
\begin{align*}
|K_\alpha(z)| \leq K_\alpha (C(\eta)|z|) 
\end{align*}
for all $z \in \Sigma_{\nicefrac{\pi}{2}-\eta}$.
\end{lem}
\begin{proof}
Since $\re(z) > 0$ for all $z \in \Sigma_{\nicefrac{\pi}{2}-\eta}$ and $\alpha \in \R$ it follows by \cite[p.~181]{Wat:95} that
\begin{align*}
|K_\alpha(z)| &= \left| \int_0^\infty e^{-z \cosh(t)} \cosh(\alpha t) \, dt \right| \leq \int_0^\infty e^{-\re(z) \cosh(t)} \cosh(\alpha t) \, dt .
\end{align*}
Note that $z = |z|e^{i\varphi}$ with $|\varphi| \in [0,\nicefrac{\pi}{2}-\eta)$. The monotony of the cosinus implies
\begin{align*}
\frac{\re(z)}{|z|} = \cos(\varphi) \geq \cos(\nicefrac{\pi}{2}-\eta) = \sin(\eta)
=: C(\eta) > 0 .
\end{align*}
Using the monotony of the exponential function and the positivity of $\cosh$ we conclude 
\begin{align*}
\int_0^\infty e^{-\re(z) \cosh(t)} \cosh(\alpha t) \, dt
\leq \int_0^\infty e^{- C(\varepsilon)|z|\cosh(t) } \cosh(\alpha t) \, dt = K_\alpha (C(\eta) |z|) 
\end{align*}
for all $z \in \Sigma_{\nicefrac{\pi}{2}-\eta}$.
\end{proof}

Therefore, we obtain an estimate for the kernel. 

\begin{lem}\label{Lem. A.2}
Let $\alpha \in \R$, $k \in [0,\infty)$ and $\lambda \in \Sigma_{\pi-\eta}$ for $\eta > 0$.
If $k + \alpha < n$, we obtain 
\begin{align*}
\sup_{x \in M}
\int_M \frac{K_{\alpha}(C(\eta) \sqrt{|\lambda|} \rho(x,y))}{\rho(x,y)^k} \, dy 
\leq C(\eta) \sqrt{|\lambda|}^{ 
k - n} 
\end{align*}
for $|\lambda|\geq 1$. 
\end{lem}
\begin{proof}
Remark that 
\begin{align*}
\int_{M}
\frac{K_{\alpha}(C(\eta) \sqrt{|\lambda|} {\rho}(x,y))}{{\rho}(x,y)^{k}} \, dy 
&= \int_{B_{R}(x)} 
\frac{K_{\alpha}(C(\eta) \sqrt{|\lambda|} {\rho}(x,y))}{{\rho}(x,y)^{k}} \, dy \\
&+ \int_{M \setminus B_{R}(x)} 
\frac{K_{\alpha}(C(\eta) \sqrt{|\lambda|} {\rho}(x,y))}{{\rho}(x,y)^{k}} \, dy .
\end{align*}
For the first term one obtains
\begin{align*}
\int_{B_{R}(x)} 
\frac{K_{\alpha}(C(\eta) \sqrt{|\lambda|} {\rho}(x,y))}{{\rho}(x,y)^{k}} \, dy
&\leq 
\tilde{C} 
\int_{{\R}^n}
\frac{K_{\alpha}(C(\eta) \sqrt{|\lambda|} |y| )}{|y|^{k}} \, dy \\
&= \hat{C}(\eta) 
\sqrt{|\lambda|}^{k}
\frac{1}{\sqrt{|\lambda|}^{n}}
\int_{{\R}^n}
\frac{K_{\alpha}(|z| )}{|z|^{k}} \, dz \\
&= \hat{C}(\eta) \sqrt{|\lambda|}^{ 
k - n} \int_0^\infty \int_{\S^{n-1}_r}  
\frac{K_{\alpha}(r)}{r^{k}} \,
\text{dvol}_{\S^{n-1}_r} \, dr \\
&= \check{C}(\eta) \sqrt{|\lambda|}^{ 
k - n} \int_0^\infty 
K_{\alpha}(r) r^{n-1-k} \, dr .
\end{align*}
Since
\begin{align*}
K_{\alpha}(r) = \mathcal{O}(r^{-\alpha})
\end{align*}
for small $r \in \R_+$ and
\begin{align*}
K_{\alpha}(r) = \mathcal{O}\left(\frac{e^{-r}}{\sqrt{r}}\right)
\end{align*}
for large $r \in \R_+$, we have
\begin{align*}
r^{n-1-k} K_{\alpha}(r) = \mathcal{O}(r^{n-1-k-\alpha})
\end{align*}
for small $r \in \R_+$ and
\begin{align*}
r^{n-1-k}  K_{\alpha}(r) = \mathcal{O}(r^{n-\nicefrac{3}{2}-k} e^{-r}) 
\end{align*}
for large $r \in \R_+$. Hence, there exists a constant $\bar{C} < \infty$ such that
\begin{align*}
\int_0^\infty 
K_{\alpha}(r) r^{n-1-k} \, dr < \bar{C} 
\end{align*}
and we conclude that
\begin{align*}
\int_{B_{R}(x)} 
\frac{K_{\alpha}(C(\eta) \sqrt{|\lambda|} {\rho}(x,y))}{{\rho}(x,y)^{k}} \, dy \leq C(\eta) \sqrt{|\lambda|}^{k-n} .
\end{align*}
If $y \in \overline{M} \setminus B_R(x)$, we have $\rho(x,y) \geq R$ 
and therefore
\begin{align*}
\int_{M \setminus B_{R}(x)} 
\frac{K_{\alpha}(C(\eta) \sqrt{|\lambda|} {\rho}(x,y))}{{\rho}(x,y)^{k}} \, dy 
&\leq \frac{K_{\alpha}(C(\eta)R \sqrt{|\lambda|})}{R^k} \text{vol}_g(M \setminus B_{R}(x)) \\
&\leq \hat{C}(\eta) e^{-\tilde{C}(\eta)\sqrt{|\lambda|}} \\
&\leq \bar{C}(\eta) \sqrt{|\lambda|}^{k-n} 
\end{align*}
for $|\lambda|$
since \begin{align*}
K_{\alpha}(r) = \mathcal{O}\left(\frac{e^{-r}}{\sqrt{r}}\right)
\end{align*}
for large $r \in \R_+$. 
\end{proof}

Replacing $x$ by $x^\ast$ this yields an estimate for the reflected kernel.

\begin{cor}\label{Cor. A.3}
Let $\alpha \in \R$, $k \in [0,\infty)$ and $\lambda \in \Sigma_{\pi-\eta}$ for $\eta > 0$. Moreover let $x \in S_{2\epsilon}$.
If $k + \alpha < n$, we obtain 
\begin{align*}
\sup_{x \in S_{2\varepsilon}} 
\int_M \frac{K_{\alpha}(C(\eta)\sqrt{\lambda} \overline{\rho}(x^\ast,y))}{\overline{\rho}(x^\ast,y)^k} \, dy  
\leq C \sqrt{|\lambda|}^{ 
k - n} 
\end{align*}
for $|\lambda|\geq 1$. 
\end{cor}



\newcommand{\etalchar}[1]{$^{#1}$}

\bigskip
\emph{Tim Binz}, University of Tübingen, Department of Mathematics, Auf der Morgenstelle 10, D-72076 Tübingen, Germany,
\texttt{tibi@fa.uni-tuebingen.de}


\begin{thebibliography}{FGG{\etalchar{+}}10}
\providecommand{\url}[1]{\texttt{#1}}
\providecommand{\urlprefix}{}

\smallskip

\bibitem[ABHN01]{ABHN:01}
W.~Arendt, C.~J.~K. Batty, M.~Hieber, and F.~Neubrander.
\newblock \emph{Vector-Valued {L}aplace Transforms and {C}auchy Problems},
  \emph{Monographs in Mathematics}, vol.~96.
\newblock Birkh\"auser (2001).
  
\

\bibitem[Ada75]{Ada:75}
R.~A.~Adams.
\newblock \emph{Sobolev Spaces},
\newblock Academic Press, New York-London (1975).

\

\bibitem[Agm62]{Agm:62}
S.~Agmon.
\newblock \emph{On th eigenfunctions and the eigenvalues of general boundary value problems.}
\newblock Comm. Pure Appl. Math. \textbf{25} (1962).

\

\bibitem[Ama95]{Ama:95}
H. Amann.
\newblock \emph{Linear and Quasilinear Parabolic Problems}, vol. 1.
\newblock Birkh\"auser (2001).

\

\bibitem[Are00]{Are:00}
W.~Arendt.
\newblock \emph{Resolvent positive operators and inhomogeneous boundary value problems}.
\newblock Ann. Scuola Norm. Sup. Pisa 24.\textbf{70} (2000), 639--670.

\

\bibitem[BE18]{BE:18}
T.~Binz and K.~Engel
\newblock \emph{Operators with Wentzell boundary conditions and the Dirichlet-to-Neumann Operator}.
\newblock Math. Nachr. (to appear 2018).

\

\bibitem[Bin18]{Bin:18}
T.~Binz
\newblock \emph{Strictly elliptic operators with Wentzell boundary conditions on spaces of continuous functions on manifolds}.
\newblock (preprint 2018).

\

\bibitem[Bro61]{Bro:61}
F.~Browder.
\newblock \emph{On the spectral theory of elliptic differential operators I}.
\newblock Math. Ann. 142.\textbf{1} (1961), 22--130.

\

\bibitem[CM98]{CM:98}
M.~Campiti and G.~Metafune.
\newblock \emph{Ventcel's boundary conditions and analytic semigroups}.
\newblock Arch. Math. \textbf{70} (1998), 377--390.

\

\bibitem[EF05]{EF:05}
K.-J. Engel and G.~Fragnelli.
\newblock \emph{Analyticity of semigroups generated by operators with
  generalized {W}entzell boundary conditions}.
\newblock Adv. Differential Equations \textbf{10} (2005), 1301--1320.

\

\bibitem[EN00]{EN:00}
K.-J. Engel and R.~Nagel.
\newblock \emph{One-{P}arameter {S}emigroups for {L}inear {E}volution
  {E}quations}, \emph{Graduate Texts in Math.}, vol. 194.
\newblock Springer (2000).

\

\bibitem[Eng03]{Eng:03}
K.-J. Engel.
\newblock \emph{The {L}aplacian on {$C(\overline\Omega)$} with generalized
  {W}entzell boundary conditions}.
\newblock Arch. Math. \textbf{81} (2003), 548--558.

\

\bibitem[Eva98]{Eva:98}
L.~C.~Evans.
\newblock \emph{Partial Differential Equations}, \emph{Graduate Studies in Mathematics.}, vol. 19.
\newblock Amer. Math. Soc. (1998).

\

\bibitem[FGG{\etalchar{+}}10]{FGGR:10}
A.~Favini, G.~R. Goldstein, J.~A. Goldstein, E.~Obrecht, and S.~Romanelli.
\newblock \emph{Elliptic operators with general {W}entzell boundary conditions,
  analytic semigroups and the angle concavity theorem}.
\newblock Math. Nachr. \textbf{283} (2010), 504--521.

\

\bibitem[FGGR02]{FGGR:02}
A.~Favini, G.~R. Goldstein, J.~A. Goldstein, and S.~Romanelli.
\newblock \emph{The heat equation with generalized {W}entzell boundary
  condition}.
\newblock J. Evol. Equ. \textbf{2} (2002), 1--19.

\

\bibitem[GT01]{GT:01}
Gilbarg, D. and Trudinger, N. S.
\newblock \emph{Elliptic partial differential equations of second order}, \emph{Classics in Mathematics}.
\newblock Springer (2001).

\

\bibitem[Heb00]{Heb:00}
E.~Hebey.
\newblock \emph{Nonlinear Analysis on Manifolds: Sobolev Spaces and Inequalities}, \emph{Courant Lecture Notes}.
\newblock Amer. Math. Soc. (2000).

\



\bibitem[Lun95]{Lun:95}
A.~Lunardi.
\newblock \emph{Analytic Semigroups and Optimal Regularity in Parabolic Problems}.
\newblock Birkh\"auser (1995).

\

\bibitem[Rud86]{Rud:86}
M.~Rudin.
\newblock \emph{Real and Complex Analysis},
\emph{Higher Mathematics Series}, vol 3.
\newblock McGraw-Hill (1986).

\

\bibitem[RR93]{RR:93}
M.~Renardy and R.~C.~Rogers.
\newblock \emph{An Introduction to Partial Differential Equations},
\emph{Texts in Appl. Math.}, vol 13.
\newblock Springer (1993).

\

\bibitem[See64]{See:64}
R.~T.~Seeley.
\newblock \emph{Extenstion of $\rC^\infty$functions defined in a half space}.
\newblock Proc. Amer. Math. Soc. \textbf{15} (1964), 625--626.

\

\bibitem[Ste74]{Ste:74}
B.~Stewart.
\newblock \emph{Generation of analytic semigroups by strongly elliptic operators}.
\newblock Trans. Amer. Math. Soc. \textbf{199} (1974), 141--161.

\

\bibitem[Wat95]{Wat:95}
G.~N.~Watson
\newblock \emph{A Treatise on the Theory of Bessel Functions},
\newblock Cambridge University Press (1995).

\end{thebibliography}
\end{document}